\documentclass[11pt]{amsart}

\usepackage{anysize} \marginsize{1.27in}{1.27in}{1in}{1in}
\usepackage{comment}
\usepackage{xcolor}
\usepackage{amsmath}
\usepackage{mathtools}
\usepackage[all]{xy}
\usepackage[utf8]{inputenc}
\usepackage{varioref}
\usepackage{amsfonts}
\usepackage{amssymb}
\usepackage{bbm}
\usepackage{esint}
\usepackage{graphicx}
\usepackage{tikz}
\usetikzlibrary{calc}
\usepackage{empheq}
\usepackage{enumitem}
\usepackage{tikz-cd}
\usetikzlibrary{matrix,arrows,decorations.pathmorphing}
\usepackage{mathrsfs}
\usepackage[hypertexnames=false,backref=page,pdftex,
 	pdfpagemode=UseNone,
 	breaklinks=true,
 	extension=pdf,
 	colorlinks=true,
 	linkcolor=blue,
 	citecolor=red,
 	urlcolor=blue,
 ]{hyperref}

\usepackage[all]{hypcap}

\newcommand{\NEbar}{{\overline{\operatorname{NE}}}}
\newcommand{\std}{{\operatorname{std}}}

\newcommand{\CK}[1]{{\operatorname{CK3}_{#1}}}

\newcommand{\CKmo}[1]{{\operatorname{CK3}_{#1}^{\scriptscriptstyle (-1)}}}

\newcommand{\MF}{{\operatorname{MF}}}

\newcommand{\Bl}{{\operatorname{Bl}}}

\newcommand{\flop}{{\operatorname{flop}}}
\newcommand{\old}{{\operatorname{old}}}
\newcommand{\new}{{\operatorname{new}}}

\newcommand{\sB}{{\mathcal B}}

\newcommand{\sO}{{\mathcal O}}

\newcommand{\sY}{{\mathcal Y}}

\newcommand{\C}{{\mathbb C}}

\newcommand{\F}{{\mathbb F}}
\newcommand{\G}{{\mathbb G}}

\renewcommand{\P}{{\mathbb P}}

\newcommand{\IS}{{\mathbb S}}

\newcommand{\Z}{{\mathbb Z}}

\newcommand{\gothn}{{\mathfrak n}}

\newcommand{\gothD}{{\mathfrak D}}

\newcommand{\gothY}{{\mathfrak Y}}

\newcommand{\isom}{{\ \cong\ }}

\newcommand{\Nef}{{\operatorname{Nef}}}

\newcommand{\ohne}{{\ \setminus \ }}

\newcommand{\Pic}{\operatorname{Pic}}

\newcommand{\ratl}{\dashrightarrow}

\newcommand{\Spec}{\operatorname{Spec}}

\renewcommand{\to}[1][]{\xrightarrow{\ #1\ }}

\newcommand{\wt}[1]{{\widetilde{#1}}}

\makeatletter
\newcommand*{\da@rightarrow}{\mathchar"0\hexnumber@\symAMSa 4B }
\newcommand*{\da@leftarrow}{\mathchar"0\hexnumber@\symAMSa 4C }
\newcommand*{\xdashrightarrow}[2][]{%
  \mathrel{%
    \mathpalette{\da@xarrow{#1}{#2}{}\da@rightarrow{\,}{}}{}%
  }%
}
\newcommand{\xdashleftarrow}[2][]{%
  \mathrel{%
    \mathpalette{\da@xarrow{#1}{#2}\da@leftarrow{}{}{\,}}{}%
  }%
}
\newcommand*{\da@xarrow}[7]{%
  \sbox0{$\ifx#7\scriptstyle\scriptscriptstyle\else\scriptstyle\fi#5#1#6\m@th$}%
  \sbox2{$\ifx#7\scriptstyle\scriptscriptstyle\else\scriptstyle\fi#5#2#6\m@th$}%
  \sbox4{$#7\dabar@\m@th$}%
  \dimen@=\wd0 %
  \ifdim\wd2 >\dimen@
    \dimen@=\wd2 %
  \fi
  \count@=2 %
  \def\da@bars{\dabar@\dabar@}%
  \@whiledim\count@\wd4<\dimen@\do{%
    \advance\count@\@ne
    \expandafter\def\expandafter\da@bars\expandafter{%
      \da@bars
      \dabar@ 
    }%
  }%
  \mathrel{#3}%
  \mathrel{%
    \mathop{\da@bars}\limits
    \ifx\\#1\\%
    \else
      _{\copy0}%
    \fi
    \ifx\\#2\\%
    \else
      ^{\copy2}%
    \fi
  }%
  \mathrel{#4}%
}
\makeatother

\makeatletter
\newsavebox\myboxA
\newsavebox\myboxB
\newlength\mylenA

\newcommand*\xtilde[2][0.8]{%
    \sbox{\myboxA}{$\m@th#2$}%
    \setbox\myboxB\null
    \ht\myboxB=\ht\myboxA%
    \dp\myboxB=\dp\myboxA%
    \wd\myboxB=#1\wd\myboxA
    \sbox\myboxB{$\m@th\widetilde{\copy\myboxB}$}
    \setlength\mylenA{\the\wd\myboxA}
    \addtolength\mylenA{-\the\wd\myboxB}%
    \ifdim\wd\myboxB<\wd\myboxA%
       \rlap{\hskip 0.5\mylenA\usebox\myboxB}{\usebox\myboxA}%
    \else
        \hskip -0.5\mylenA\rlap{\usebox\myboxA}{\hskip 0.5\mylenA\usebox\myboxB}%
    \fi}

\newbox\usefulbox

\def\getslant #1{\strip@pt\fontdimen1 #1}

\def\xxtilde #1{\mathchoice
 {{\setbox\usefulbox=\hbox{$\m@th\displaystyle #1$}%
    \dimen@ \getslant\the\textfont\symletters \ht\usefulbox
    \divide\dimen@ \tw@ 
    \kern\dimen@ 
    \xtilde{\kern-\dimen@ \box\usefulbox\kern\dimen@ }\kern-\dimen@ }}
 {{\setbox\usefulbox=\hbox{$\m@th\textstyle #1$}%
    \dimen@ \getslant\the\textfont\symletters \ht\usefulbox
    \divide\dimen@ \tw@ 
    \kern\dimen@ 
    \xtilde{\kern-\dimen@ \box\usefulbox\kern\dimen@ }\kern-\dimen@ }}
 {{\setbox\usefulbox=\hbox{$\m@th\scriptstyle #1$}%
    \dimen@ \getslant\the\scriptfont\symletters \ht\usefulbox
    \divide\dimen@ \tw@ 
    \kern\dimen@ 
    \xtilde{\kern-\dimen@ \box\usefulbox\kern\dimen@ }\kern-\dimen@ }}
 {{\setbox\usefulbox=\hbox{$\m@th\scriptscriptstyle #1$}%
    \dimen@ \getslant\the\scriptscriptfont\symletters \ht\usefulbox
    \divide\dimen@ \tw@ 
    \kern\dimen@ 
    \xtilde{\kern-\dimen@ \box\usefulbox\kern\dimen@ }\kern-\dimen@ }}%
 {}}

\newcommand*\xoverline[2][0.75]{%
    \sbox{\myboxA}{$\m@th#2$}%
    \setbox\myboxB\null
    \ht\myboxB=\ht\myboxA%
    \dp\myboxB=\dp\myboxA%
    \wd\myboxB=#1\wd\myboxA
    \sbox\myboxB{$\m@th\overline{\copy\myboxB}$}
    \setlength\mylenA{\the\wd\myboxA}
    \addtolength\mylenA{-\the\wd\myboxB}%
    \ifdim\wd\myboxB<\wd\myboxA%
       \rlap{\hskip 0.5\mylenA\usebox\myboxB}{\usebox\myboxA}%
    \else
        \hskip -0.5\mylenA\rlap{\usebox\myboxA}{\hskip 0.5\mylenA\usebox\myboxB}%
    \fi}

\def\xxoverline #1{\mathchoice
 {{\setbox\usefulbox=\hbox{$\m@th\displaystyle #1$}%
    \dimen@ \getslant\the\textfont\symletters \ht\usefulbox
    \divide\dimen@ \tw@ 
    \kern\dimen@ 
    \overline{\kern-\dimen@ \box\usefulbox\kern\dimen@ }\kern-\dimen@ }}
 {{\setbox\usefulbox=\hbox{$\m@th\textstyle #1$}%
    \dimen@ \getslant\the\textfont\symletters \ht\usefulbox
    \divide\dimen@ \tw@ 
    \kern\dimen@ 
    \xoverline{\kern-\dimen@ \box\usefulbox\kern\dimen@ }\kern-\dimen@ }}
 {{\setbox\usefulbox=\hbox{$\m@th\scriptstyle #1$}%
    \dimen@ \getslant\the\scriptfont\symletters \ht\usefulbox
    \divide\dimen@ \tw@ 
    \kern\dimen@ 
    \xoverline{\kern-\dimen@ \box\usefulbox\kern\dimen@ }\kern-\dimen@ }}
 {{\setbox\usefulbox=\hbox{$\m@th\scriptscriptstyle #1$}%
    \dimen@ \getslant\the\scriptscriptfont\symletters \ht\usefulbox
    \divide\dimen@ \tw@ 
    \kern\dimen@ 
    \xoverline{\kern-\dimen@ \box\usefulbox\kern\dimen@ }\kern-\dimen@ }}%
 {}}
\makeatother

\makeatletter
\newcommand{\mylabel}[2]{#2\def\@currentlabel{#2}\label{#1}}
\makeatother

\makeatletter
\newcommand{\Mac}{}
\DeclareRobustCommand{\Mac}{%
  M%
  \raisebox{\dimexpr\fontcharht\font`M-\height}{%
    \check@mathfonts\fontsize{\sf@size}{0}\selectfont
    c%
  }%
}
\makeatother

\newtheoremstyle{citing}
  {}
  {}
  {\itshape}
  {}
  {\bfseries}
  {\textbf{.}}
  {.5em}
  {\thmnote{#3}}

\theoremstyle{plain}
\newtheorem{theorem}{Theorem}

\newtheorem{lemma}[theorem]{Lemma}
\newtheorem{corollary}[theorem]{Corollary}

\newtheorem{setup}[theorem]{Setup}

\newtheorem*{maintheorem}{Main Theorem}

\newtheorem{proposition}[theorem]{Proposition}

\theoremstyle{remark}
\newtheorem{example}[theorem]{Example}

\theoremstyle{definition}

\newtheorem{definition}[theorem]{Definition}
\newtheorem{construction}[theorem]{Construction}
\newtheorem{notation}[theorem]{Notation}

\numberwithin{equation}{section}

\theoremstyle{remark}
\newtheorem{remark}[theorem]{Remark}

{\theoremstyle{citing}
}

{\theoremstyle{definition}
}

\title[Combinatorial K3 surfaces]{Combinatorial K3 surfaces and the Mori fan of the Dolgachev--Nikulin--Voisin family in degree 2}

\author{Klaus Hulek}
\address{Klaus Hulek\\Institut f\"ur Algebraische Geometrie, Leibniz University Hannover, Welfengarten 1, 30167 Hannover,
Germany}
\email{hulek@math.uni-hannover.de}

\author{Christian Lehn}
\address{Christian Lehn\\ Fakult\"at f\"ur Mathematik\\Ruhr-Universit\"at Bochum\\  Universit\"ats\-stra\ss e~150\\ Postfach IB 45\\ 44801 Bochum, Germany}
\email{christian.lehn@rub.de}

\let\origmaketitle\maketitle
\def\maketitle{
  \begingroup
  \def\uppercasenonmath##1{} 
  \let\MakeUppercase\relax 
  \origmaketitle
  \endgroup
}

\numberwithin{theorem}{section}

\def\nodesize{9pt}

\tikzset{
  vertex/.style={circle, draw, inner sep=0pt, minimum size=\nodesize},
  special/.style={circle, fill=black, inner sep=0pt, minimum size=\nodesize},
  bluewhite/.style={circle, draw=blue, fill=blue!50, fill opacity=0.5, inner sep=0pt, minimum size=\nodesize},
  bluefill/.style={circle, fill=blue, inner sep=0pt, minimum size=\nodesize}
}

\begin{document}
\thispagestyle{empty}

\begin{abstract}
We introduce the notion of a combinatorial K3 surface. Those form a certain class of type III semistable K3 surfaces and are completely determined by combinatorial data called curve structures. Emphasis is put on degree $2$ combinatorial K3 surfaces, but the approach can be used to study higher degree as well. We describe elementary modifications both in terms of the curve structures as well as on the Picard groups. Together with a description of the nef cone in terms of curve structures, this provides an approach to explicitly computing the Mori fan of the Dolgachev--Nikulin--Voisin family in degree $2$.
\end{abstract}


\maketitle

\setlength{\parindent}{1em}
\setcounter{tocdepth}{1}



\tableofcontents

\section{Introduction}\label{section intro}

The Mori fan of Dolgachev--Nikulin--Voisin (DNV) families is central in the Gross--Hacking--Keel--Siebert approach to compactifying moduli spaces using mirror symmetry. For a model $\sY$ of the DNV family, that is, a certain kind of semistable type III degeneration of K3 surfaces, the Mori fan $\MF(\sY)$ is defined in terms of the birational geometry of the threefold $\sY$. 

The philosophy of this article is that most questions about $\MF(\sY)$ can be answered in terms of the combinatorics of the irreducible components of the central fiber $\sY_0$. We will thus replace the smooth three-dimensional scheme $\sY$ by a reducible surface with rational components. These are special {$d$-semistable K3 surfaces of type III} which we will refer to as \emph{combinatorial K3 surfaces}. Their combinatorial nature is based on the fact that they are determined by their \emph{curve structures}, an invariant introduced in~\cite{HL22}. Our goal is to utilize these surfaces to gain explicit control over the Mori fan. This is carried out for surfaces of degree $2$.

Projective models of combinatorial K3 surfaces in degree $2$ are all related by so-called type I and type II flops. Our main contribution is  a combinatorial description of these flops at the level of Picard groups. In the above notation, the pullback gives an embedding of $\Pic(\sY)$ to the Picard group of the normalization $\sY_0^\nu$ of the central fiber. Given a flop between $\sY \ratl \sY'$ between models, we define a map $\phi_c$ between the Picard groups of the normalizations of the central fibers in terms of their curve structures. This map is referred to as the \emph{combinatorial pushforward}, see Definitions~\ref{definition type i flop on picard of normalizations} and~\ref{definition type ii flop on picard of normalizations}. Then our main result can be stated as follows, see Theorems~\ref{theorem type i flop on picard} and~\ref{theorem type ii flop on picard}.

\begin{maintheorem}
Let $\Phi:\sY \ratl \sY'$ be a type I or type II flop between models of the DNV family in degree $2$. Then the combinatorial pushforward $\phi_c:\Pic(\sY_0^\nu) \to \Pic(\sY_0^{\prime\nu})$ restricts to $\Phi_*$ on $\Pic(\sY)$. 
\end{maintheorem}

While this description is surprisingly simple, it requires us to develop a lot of structure theory of combinatorial K3 surfaces. This is a natural continuation of \cite{HL22} and builds strongly on it.

With our main theorem at hand, one can map all cones in the Mori fan to the Picard group of a fixed reference model. Together with a description of the cone of curves of a combinatorial K3 surface in terms of its curve structure (which thus provides us with an explicit description of its nef cone), this provides an approach to explicitly computing the Mori fan of the Dolgachev--Nikulin--Voisin family in degree $2$.

Based on our description, one can implement the Mori fan on a computer. This was the original motivation for this work, but in the course of the project, it became clear that there were theoretical questions to be answered that are equally interesting. Apart from the main theorem, this includes e.g.\ a better understanding of interior special points as in Section~\ref{section interior special points} or an explicit description of the cone of curves in terms of the curve structure, see Proposition~\ref{proposition cone generators}. As a follow-up project, we have implemented combinatorial K3 surfaces in SageMath \cite{sagemath} and will use it in forthcoming work to further study the Mori fan and related constructions. 

\subsection*{Outline of the paper} 
In Section~\ref{section combinatorial k3 surfaces}, we introduce the class of combinatorial K3 surfaces and discuss their relation to models of the DNV family and their basic properties. We recall the notion of a curve structure from \cite{HL22} and define type I flops in the context of combinatorial K3 surfaces. A flop (of type I or II) is not different from the effect that a corresponding flop between models of the DNV family has on its central fibers. However, when exclusively working with these singular surfaces, an additional subtlety arises. To get from the normalization to the surface itself, one has to specify a gluing of the double curves. This is not an issue for type I flops, since they are isomorphisms at the generic point of the double curves so that the gluing is not affected. However, for type II flops new boundary components are introduced and others disappear so that the gluing after the flop has to be discussed. This is done via interior special points, which are treated in Section~\ref{section interior special points}. It turns out that the appearance of interior special points after a blow-up can be detected by special rulings on the irreducible components of the combinatorial K3 surfaces. While those points are determined by the Carlson map in all degrees, we were not able to turn the latter into an explicit description. Substituting the Hodge theoretic ingredient of the Carlson map by the geometric approach via rulings limits our description to degree $2$ here, but this is enough for the purpose of this paper. The constructive approach to interior special points is used in Section~\ref{section type ii flops} to specify the gluing for the newly appearing boundary components in a type II flop between degree $2$ combinatorial K3 surfaces. The effect of a type II flop on the curve structures is discussed. In Section~\ref{section bases and cones}, we provide several results needed for computing the Mori fan explicitly. Firstly, we explain the effect of a type I flop on the curve structure in Section~\ref{section type i flops on curve structures}. Then, we prove in Lemma~\ref{lemma curve structure is z basis} that the curves in the curve structure form a $\Z$-basis of the Picard group of the normalization. Finally, we show that the curves in the augmented curve structure generate the cone of curves, see Proposition~\ref{proposition cone generators}.

The final section, Section~\ref{section flops on picard}, defines the combinatorial pushforward for flops of type I and II and proves the main theorem, all in the degree $2$ case.

\subsection*{Notation} We will abbreviate $S:=\Spec\C[[t]]$. Schemes over $S$ will be denoted by caligraphic letters like $\sY\to S$ whereas straight letters will usually refer to $\C$-schemes. For a variety $X$, we will denote by $\nu:X^\nu \to X$ its normalization.

\subsection*{Acknowledgements} 
Klaus Hulek was partially supported by DFG grant \mbox{Hu~337/7-2}. Christian Lehn was
supported by the DFG through the research grants Le 3093/3-2 and Le 3093/5-1. We thank the referee for their thorough reading and valuable comments.
 
\section{Combinatorial K3 surfaces}\label{section combinatorial k3 surfaces}

\subsection{Models of the DNV family}\label{section models}

Recall from \cite[Definition~2.22]{HL22} that a model of the Dolgachev--Nikulin--Voisin family (or DNV family) of degree $2d$ is a primitive projective type III degeneration $\sY\to S:=\Spec\C[[t]]$ which is maximal and whose generic fiber has Picard group isomorphic to 
\[
\check M_{2d}:=U\oplus E_8(-1)^{\oplus 2}\oplus \left\langle -2d\right\rangle.
\]
The notions of a type III degeneration can be found in Definition~2.2, of primitivity in Definition~2.8, and of maximality in Definition~2.15 of \cite{HL22}. The precise knowledge of these terms will however not be essential for the present paper, as we use the rather concrete descriptions in terms of their central fibers. Indeed, by \cite[Proposition~2.13]{HL22} models of the DNV family are determined uniquely by their central fibers, which are essentially combinatorial objects. Let us recall the following notion.

\begin{definition}\label{definition d semistable type iii k3 surface}
A \emph{$d$-semistable K3 surface of type III} is a normal crossing surface $Y = \bigcup Y_i$ such that
\begin{enumerate}
    \item $Y$ is $d$-semistable in the sense of Friedman \cite[Definition~1.13]{Frie83}.
    \item The dualizing sheaf $\omega_Y = \mathcal{O}_Y$ is trivial.
    \item The irreducible components $Y_i$ are rational surfaces, and double curves give rise to anticanonical
    cycles of rational curves on the normalization~$Y_i^{\nu}$. 
    \item The dual intersection complex of $Y$ is a triangulation\footnote{For a triangulation, we allow there to be several $p$-simplices on a given set of $p+1$ vertices.} of the sphere $\IS^2$.
\end{enumerate}
\end{definition}

\begin{remark}\label{remark log calabi yau surfaces}
For an equidimensional $S_2$-scheme $X$ with trivial dualizing sheaf $\omega_X$, the conductor ideal defines an anticanonical cycle on the normalization by subadjunction, see \cite[Proposition~2.3]{Rei94}. In particular, for a $d$-semistable K3 surface of type III the preimages of the double curves on the normalizations $Y_i^\nu\to Y_i$ are anticanonical cycles, i.e. the double curves give rise to log Calabi--Yau surface pairs $(Y_i^\nu,D_i)$. As \'etale covers of (possibly singular) rational curves, the preimages are again rational.
\end{remark}

Let $Y$ be a $d$-semistable K3 surface of type III and denote by $Y_i^\nu \to Y_i$ the normalizations of its irreducible components. Let us consider the canonical sequence 
\begin{equation}\label{eq maximality sequence}
0 \to \Pic(Y) \to \bigoplus_i H^2(Y_i^\nu, \mathbb{Z}) \to \bigoplus_C H^2(C^\nu, \mathbb{Z})
\end{equation}
where $C$ runs through all double curves, see equation (2.3) on p.9 of \cite{HL22}. Note that the curve $C$ might get normalized in $Y_i^\nu$, so it is more natural to write $C^\nu$ in the last term, even if it does not change anything on the level of $H^2$. The following definition is central to this article.

\begin{definition}\label{definition combinatorial k3 surface}
A $d$-semistable K3 surface $Y$ of type III is called \emph{maximal} if the sequence \eqref{eq maximality sequence} is exact. A \emph{combinatorial K3 surface} is a maximal $d$-semistable K3 surface. Such a surface is said to have degree $2d$ if it has $d+2$ irreducible components.
\end{definition}

\noindent Note that the notion of maximality from the previous definition coincides with the one from \cite[Definition~2.15]{HL22}. This is shown in \cite[Lemma~2.10]{HL22}, taking into account equation (2.4) in op. cit. In particular, the central fiber of a type III degeneration is a maximal $d$-semistable K3 surface if and only if the degeneration itself is maximal.

\begin{remark}\

\begin{enumerate}
\item Equivalently, we can describe a combinatorial K3 surface as follows. Instead of a surface $Y$ as above, we describe it as a semisimplicial object $\xymatrix@C=1.5em{Z_1 \ar@<2pt>[r] \ar@<-2pt>[r] & Z_0}$ of length $2$. Here, $Z_0=Y^\nu$ is the normalization and $Z_1$ is the disjoint union over all double curves. One can think of this as a gluing datum, i.e. $Y$ as the pushout of this diagram. Note that pushouts always exist for finite morphisms by \cite[Théorème~5.4]{Fer03}.
	\item By \cite[Lemma~2.11]{HL22}, the maximality property rigidifies the notion of a $d$-semistable K3 surface of type III in the sense that it singles out a distinguished point in its space of locally trivial deformations. This is essential for fully reducing a degeneration to a combinatorial object, see also Remark~\ref{remark combinatorial nature}.
	\item The definition of the degree of a combinatorial K3 surface is in line with the definition of the degree of a model of the DNV family in \cite{HL22}. By \cite[Proposition~2.13]{HL22}, the central fiber of a model of the DNV family in degree $2d$ yields a triangulation of $\IS^2$ with $2d$ triangles. Note that the index $k$ showing up in loc. cit. is equal to $1$ by primitivity, see \cite[Definition~2.8]{HL22}. Hence, this makes $3d$ edges and hence $d+2$ vertices by Euler's theorem.
	\item Also, the primitivity property of models of the DNV family is defined for the central fiber of the model, see~\cite[Section~2]{HL22}. In fact, it is a property of the dual intersection complex of this combinatorial K3 surface and we do not need the model to define it. It seemed more natural to us, not to include it in the definition of a combinatorial K3 surface because it does not contribute to the combinatorial nature of this object and anyway in this article we are mainly concerned with the very explicit models in degree two.
\end{enumerate}
\end{remark}

We introduce some notation for combinatorial K3 surfaces.

\begin{notation}
Let $d$ be a positive integer and $G$ a triangulation of the sphere $\IS^2$. We introduce the following classes of objects.
\begin{description}[style=multiline, labelwidth=1.8cm, labelsep=0.5cm, leftmargin=3cm] 
\item[$\CK{}$] The class of combinatorial K3 surfaces.
\item[$\CKmo{}$] The subclass of $\CK{2d}$ consisting of surfaces in $(-1)$-form.
\end{description}
We write $\CK{2d}$ and $\CKmo{2d}$ for the subclasses of surfaces of degree $2d$.
\end{notation}

\begin{definition}
A combinatorial K3 surface $Y$ is said to be in $(-1)$-\emph{form} if the following holds. For any integral curve $D\subset Y^\nu$ in the preimage of the double locus, we have $D^2 = -1$ if $D$ is smooth and $D^2=1$ if $D$ is singular (and hence a nodal rational curve), cf. the corresponding notion for models in \cite[Definition~2.6]{HL22}.
\end{definition}

The following proposition summarizes \cite[Lemma~2.11]{HL22} and \cite[\S~5.1]{Laz08}.

\begin{proposition}\label{proposition minus one form triangulation}
Taking the dual intersection complex defines a bijection between the set of isomorphism classes of combinatorial K3 surfaces in $(-1)$-form and triangulations of the sphere $\IS^2$ such that no vertex has valency greater than~$6$. In particular, a combinatorial K3 surface in $(-1)$-form is uniquely determined by its dual intersection complex.\qed
\end{proposition}

\begin{definition}\label{definition minus one form triangulation}
Let $G$ be a triangulation of the sphere $\IS^2$. We denote by $Y_G$ the unique combinatorial K3 surface in $(-1)$-form with dual intersection complex $G$.
\end{definition}

\begin{remark}\label{remark combinatorial nature}
The classification of combinatorial K3 surfaces in $(-1)$-form through special triangulations of the sphere	from Proposition~\ref{proposition minus one form triangulation} is a first instance of the combinatorial nature of combinatorial K3 surfaces. But not only those in $(-1)$-form are determined by combinatorial objects (at least in degree $2$, even though this is likely to hold true in higher degree as well): a general combinatorial K3 surface is determined by its \emph{curve structure}, a certain graph that captures the configuration of curves on the surface, see Definition~\ref{definition curve structure}. This fully justifies the terminology \emph{combinatorial K3 surface}. 
\end{remark}

\subsection{Type I flops}\label{section type i flops}

The Mori fan is defined in terms of the birational geometry of models of the DNV family. Flops between these models can be factored into special flops, which are called type I and type II flops respectively, 
and automorphisms, see \cite[Corollary~6.34]{HL22}. Recall that our approach is to forget about the model and only deal with the associated combinatorial K3 surface. The aforementioned flops induce certain birational modifications on the combinatorial K3 surfaces which are called \emph{elementary modifications of type I respectively II}. For simplicity, we will abuse the terminology a little and refer to them as type I and type II flops as well, even though they are not technically flops. We will discuss type I flops in this section. Type II flops require some more preparation and will be treated in Section~\ref{section type ii flops}.

\begin{notation}\label{notation double curves}
Let $Y$ be a combinatorial K3 surface, let $Y_i\neq Y_j$ be irreducible components of $Y$, and let $C$ be a curve in their intersection. Then the preimage of $C$ under the normalization has two one-dimensional irreducible components, one in each of the connected components $Y_i^\nu$ and $Y_j^\nu$ of $Y^\nu$. We denote them by $D_{ij}\subset Y_i^\nu$ and $D_{ji} \subset Y_j^\nu$. 
For $i=j$, that is, a component $Y_i=Y_j$ with self-intersection containing a curve $C$, we write $D_{ii}^1 \cup D_{ii}^2 \subset Y_{i}^\nu$ for the one-dimensional irreducible components in the preimage of $C$. 
\end{notation}

We will give an abstract definition of type I and type II flops for combinatorial K3 surfaces, which is a translation of the usual notion of type I and II flops, see \cite[p.~8]{HL22}, to the class of combinatorial K3 surfaces. Here, we shall restrict ourselves to degree 2. We will start with type I flops, see Definition~\ref{def:type I flops final}. Type II flops will be treated in Section~\ref{section type ii flops} after interior special points have been discussed in Section~\ref{section interior special points}.

\begin{construction}\label{construction type i flop}
Assume that for some $i$ we have an irreducible curve $C_i\subset Y_i$ not contained in the double locus, whose preimage $C_i^\nu$ in the normalization $Y_i^\nu$ is a $(-1)$\nobreakdash-curve. Recall from Remark~\ref{remark log calabi yau surfaces} that the double curves give rise to an anticanonical cycle on any combinatorial K3 surface. Hence, by adjunction, $C_i^\nu$ intersects the double locus in a single point $P_i \in D_{ij}$, which is not a triple point, and the intersection is transversal. Here $i=j$ is allowed, in which case the assumption is that in the normalization, $C_i^{\nu}$ meets one of the components of the preimage $D_{ii}^1 \cup D_{ii}^2$, say $D_{ii}^1$. We then perform the following construction. 

If $i \neq j$,  we set $Y'^{\nu}_k:=Y^{\nu}_k$ for $k \neq i,j$. We have also natural identifications\footnote{Let us point out that these identifications will preserve so-called interior special points, see Definition~\ref{definition interior special points}. This is however not relevant for the construction of the flop in the type I case.} of the components $D_{ki}, D_{kj}$ and $D'_{ki}, D'_{kj}$ of the components of the anticanonical cycle.  We define $Y'^{\nu}_i$ as the blow-down of $Y_i^{\nu}$ along the preimage of the $(-1)$-curve $C_i$. We further define $Y'^{\nu}_j$ as the blow-up of $Y_j$ in the preimage\footnote{This will be called the interior special point on $D_{ji}$ in the next section.} of the point where $C_i$ meets $Y_j$.

If $i=j$, we perform the analogous construction on $Y^{\nu}_i$, i.e. blow-down along the preimage of $C_i$ and blow-up in the point of $D_{ii}^2$ lying over $P_i$. We have now constructed a union $Y^{\prime\nu}:=\cup_i Y^{\prime\nu}_i$ and a birational map $\phi:Y^\nu \ratl Y^{\prime\nu}$ that restricts to a blow-down respectively a blow-up $Y_i^\nu \ratl Y^{\prime \nu}_i$ on a given component. Notice that this is actually abuse of notation, for so far we have not yet defined a surface $Y'$ that $Y^{\prime\nu}$ would be the normalization of (and similarly for the components $Y_i^{\prime\nu}$). This comes next: Since $\phi$ is an isomorphism at the generic point of each double curve $D_{k\ell}$, the gluing of the double curves along $Y^\nu \to Y$ induces in a unique way a gluing $Y^{\prime\nu} \to Y'$ to a singular surface $Y'$ and $Y^{\prime\nu}$ is identified with the normalization of $Y$.  We note that by construction, the Carlson homomorphism \cite[p. 10]{HL22} of $Y'$ is trivial, $c_{Y'}=1$. Hence, $Y'$ is maximal and thus a combinatorial K3 surface.
\end{construction}

We depict this construction in Figure \ref{figure type i flop}.
            
\begin{definition} \label{def:type I flops final}
We say that the combinatorial K3 surface $Y'$ from Construction~\ref{construction type i flop} {\em arises from $Y$ by the type I flop in the curve $C_i$} and we refer to the birational map $\phi: Y \dashrightarrow Y'$ as a 
{\em type I flop} of combinatorial K3 surfaces.   
\end{definition}  

\begin{remark}
The type I flops described here are the restrictions of type I flops of maximal Kulikov models. Indeed, let $Y \subset \sY$ where $\sY$ is the unique maximal smoothing, see \cite[Proposition 2.13]{HL22}. 
Then we can perform a type I flop along $C_i$ resulting in a birational map $\tilde \phi: \sY \dashrightarrow \sY''$ with special fiber $Y''$. This is again a maximal degeneration by \cite[Proposition 2.16]{HL22}. By inspection, there is a natural isomorphism $Y'' \cong Y'$ and thus a natural isomorphims $\sY'' \cong  \sY'$ which shows that $\phi$ is indeed the restriction of a type I flop of maximal Kulikov models. 
\end{remark}

\begin{figure}\centering
\begin{tikzpicture}[scale=0.6]
\tikzset{mygreen/.style={green!66!black}}
  \draw (-1,0) -- (-1,4);
  \draw (0,0) -- (0,4);
  \draw[color=red] (-1,2) -- (-4,2);
  \node[above] at (-1.2,4) {$D_{ij}$};
  \node[above] at (0.2,4) {$D_{ji}$};
  \node at (-3.3,2.8) {$C_i$};

  \fill[mygreen] (-1,2) circle (3pt);
  \node[below left] at (-1,2) {$P_i$};
  \fill[mygreen] (0,2) circle (3pt);
  \node[below right] at (0,2) {$P_j$};

  \begin{scope}[xshift=10cm]
    \draw (-1,0) -- (-1,4);
    \draw (0,0) -- (0,4);
    \draw[color=red] (0,2) -- (3,2);
    \node[above] at (-1.2,4) {$D'_{ij}$};
    \node[above] at (0.2,4) {$D'_{ji}$};
    \node at (2.3,2.8) {$C'_j$};

    \fill[mygreen] (-1,2) circle (3pt);
    \node[below left] at (-1,2) {$P_i'$};
    \fill[mygreen] (0,2) circle (3pt);
    \node[below right] at (0,2) {$P_j'$};
  \end{scope}

  \draw[->, dashed, thick] (2.8,2) -- (6.3,2)
    node[midway, above] {$\phi^{\nu}$};
\end{tikzpicture}
\caption{A type I flop (on the normalizations) of combinatorial K3 surfaces.}
\label{figure type i flop}
\end{figure}
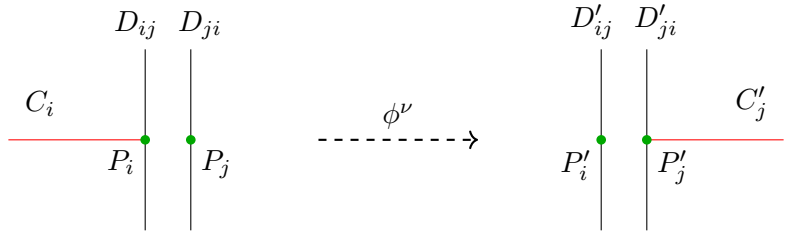

\subsection{Combinatorial K3 surfaces of degree \texorpdfstring{$2$}{2}}\label{section degree two}

In this article, we will mainly be concerned with combinatorial K3 surfaces of degree $2d=2$. Such surfaces have $d+1=3$ irreducible components, and the dual intersection complex is one of the two triangulations of $\IS^2$ with three vertices; see Figure~\ref{figure p and t}. As indicated in the figure, we will denote these triangulations by $P$ respectively $T$. Depending on its dual intersection complex, a combinatorial K3 surface of degree $2$ is thus said to be \emph{of class} $P$ resp. $T$.

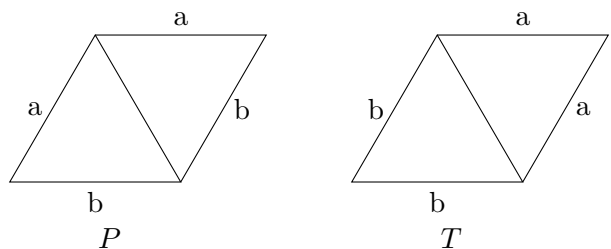
\begin{figure}[!ht]%
\begin{tikzpicture}[scale=0.75]
\draw[]    (0,1) -- ++(0:1.5) node(C1)[below]{b}--++(0:1.5)  
-- ++(120:1.5) node(A1){}--++ (120:1.5) node(top1){} --++ (240:1.5)node(B1)[left]{a} --cycle; 
\draw[] (top1.center)--++(0:1.5) node(C2)[above]{a}--++(0:1.5) --++(-120:1.5) node[right]{b}--++(-120:1.5);
\draw node at (1.75,0) {${P}$};

\draw[]    (6,1) -- ++(0:1.5) node(C3)[below]{b}--++(0:1.5)  
-- ++(120:1.5) node(A3){}--++ (120:1.5) node(top2){} --++ (240:1.5)node(B3)[left]{b} --cycle; 
\draw[] (top2.center)--++(0:1.5) node(C4)[above]{a}--++(0:1.5) --++(-120:1.5) node[right]{a}--++(-120:1.5);
\draw node at (7.75,0) {${T}$};
\end{tikzpicture}
\caption{The triangulations $P$ and $T$ of $\IS^2$.}%
\label{figure p and t}%
\end{figure}

\begin{remark}
Type I flops preserve the class of the combinatorial K3 surface, i.e. $Y$ is of class $T$ (or $P$) if and only if $Y'$ is.  
\end{remark} 

By Proposition~\ref{proposition minus one form triangulation}, there are two combinatorial K3 surfaces of degree $2$ in $(-1)$-form, one each of class $P$ resp. $T$. Following \cite[Section~2.2]{HL22}, we call them $Y_P$ and $Y_T$. Let us recall their construction and some basic facts about their geometry.

\begin{example}\label{example gothy}
By \cite[Proposition 5.2, Lemma 5.14]{Laz08} (see also \cite[Construction~2.25]{HL22} for more details and precise references), there are weak del Pezzo surfaces $\gothY_i$, $i=1,\ldots,6$ such that any connected component of the normalization of any $Y \in \CKmo{}$ is isomorphic to one of those. 
More precisely, the $\gothY_i$ come with an anticanonical cycle $\gothD_i$ having $i$ irreducible components and the components of the normalization of $Y$ are isomorphic to $(\gothY_i,\gothD_i)$ as log Calabi--Yau surfaces, see Remark~\ref{remark log calabi yau surfaces}. While the details of their construction can be found in the above references, we are mainly interested in their combinatorial description via curve structures. The latter will be defined in general in Section~\ref{section curve structures}, but for the surfaces $\gothY_i$ it is straightforward: we define the set
\[ 
C(\gothY_i):=\{ C\subset \gothY_i \mid C \text{ is an integral curve with } C^2<0, C\not\subset \gothD_i \}. 
\]
The \emph{curve structure} $\Gamma_{\gothY_i}$ of $\gothY_i$ is now defined as the dual graph of $C(\gothY_i)$ where the vertices are labeled by the self-intersection numbers of the corresponding curves. We will write $C_v$ for the curve corresponding to a vertex $v\in \Gamma_{\gothY_i}$. For each component $D$ of $\gothD_i$, we append a vertex $d$ to $\Gamma_{\gothY_i}$ and for each $v\in \Gamma_{\gothY_i}$ we add an edge joining $d$ and $v$ if $D.C_v \neq 0$. The labeled graph $\Gamma_{\gothY_i}^a$ thus obtained is called the \emph{augmented} curve structure of $\gothY_i$. 
\end{example}

For us, only the surfaces $\gothY_1$, $\gothY_2$, and $\gothY_4$ will play a role. We will next describe them geometrically. To this end, we recall the notion of an $n$-blow-up from \cite[Definition~2.24]{HL22}, for the reader's convenience. The reader may also find it helpful to consult the references from the preceding example. We remark that the geometric description of $\mathfrak{Y}_4$ given here coincides with that in \cite[Construction~2.25]{HL22}, whereas the descriptions of $\mathfrak{Y}_1$ and $\mathfrak{Y}_2$ differ, but the 
surfaces are the same.

\begin{definition}\label{def n blow up}
Let $(Y,D)$ be an anticanonical pair and let $p$ be a smooth point of $D$. If $n = 1$, the \emph{$n$-blow-up} of $Y$ in $p$ is the usual blow-up. For $n > 1$, the $n$-blow-up of $Y$ in $p$ is defined inductively as the blow-up of the $(n - 1)$-blow-up $\pi \colon Y' \to Y$ at the intersection point of the strict transform of $D$ with the exceptional set of $\pi$.

More generally, for an ordered set $(p_1, \ldots, p_k)$ of points lying in the smooth locus of $D$ and a $k$-tuple $n=(n_1,\ldots,n_k)$ of nonnegative integers, the \emph{$n$-blow-up} of $Y$ in $(p_1,\ldots,p_k)$ is defined by applying the above construction independently at each $p_j$, with multiplicity $n_j$.
\end{definition}

\begin{example}\label{example geometric description gothy1}
The surface $\mathfrak{Y}_1$ is obtained by performing the $8$-blow-up of~$\mathbb{P}^2$ in an inflection point of a nodal cubic in the sense of Definition~\ref{def n blow up}, and the boundary divisor $\mathfrak{D}_1$ is the strict transform of that cubic. This can be read off from Figure~\ref{figure gothy1 curve structure}: After blowing down the eight curves that form the horizontal chain in the curve structure, we obtain a smooth rational surface of Picard rank one. It must therefore be isomorphic to $\mathbb{P}^2$. The strict transform of $\mathfrak{D}_1$ is a nodal cubic, and the remaining curve from the curve structure becomes a flex line of this nodal cubic. 

Note that the construction is independent of the choice of a flex line. A nodal cubic has three inflection points (see e.g. \cite[Proposition~9 (2)]{Oka07}), but the group of automorphisms of $\P^2$ that fix the cubic acts transitively on the flex lines. This can easily be seen in coordinates: We can assume that the cubic is given by the equation
\[
x^3+y^3-3xyz=0
\]
A straightforward calculation gives that the flex points are $[1:-1:0]$, $[1:-\omega:0]$, and $[1:-\omega^2:0]$ where $\omega$ denotes a primitive third root of unity. They form one orbit of the cyclic group generated by the  automorphism $[x:y:z]\mapsto [x:\omega y:\omega^2 z]$.
\end{example}

\begin{example}\label{example geometric description gothy2}
We obtain the surface $\mathfrak{Y}_2$ from the Hirzebruch surface $\mathbb{F}_2$ together with two distinct sections \mbox{$S_1$, $S_2$} of the ruling that are smooth curves of self-intersection $S_1^2=S_2^2=2$ and intersect in two distinct points. One selects a fiber of the ruling that does not meet the two points of $S_1 \cap S_2$ and performs a $(3,3)$-blow up in the two points where the fiber meets $S_1$ and $S_2$. The resulting surface is $\mathfrak{Y}_2$, and $\mathfrak{D}_2$ is defined as the strict transform of the union of $S_1$ and $S_2$. 

Again, this construction is independent of the choices made. A convenient model to verify this claim is the cone in $\P^3$ over a smooth plane conic. This cone is also the image of the map given by the linear system $|C_0 + 2f|$ where $C_0 \subset \F_2$ denotes the unique $(-2)$-curve and $f$ denotes a fiber of the ruling. Note that $S_1,S_2 \in |C_0 + 2f|$, i.e. they are obtained as hyperplane sections of the cone that do not meet the vertex. The automorphism group of $\P^3$ fixing the cone (equivalently, that of $\F_2$) acts transitively on ordered pairs of such hyperplanes. The stabilizer of such a pair acts transitively on the complement of the image of $S_1\cap S_2$ in the base curve of the ruling. Note that the latter is isomorphic to $\P^1\setminus\{0,\infty\}$ and the stabilizer contains a copy of $\G_m$ acting in the standard way on this.
\end{example}

\begin{example}\label{example geometric description gothy4}
To construct $\mathfrak{Y}_4$, one starts with the product surface $\mathbb{P}^1 \times \mathbb{P}^1$, together with two distinct lines $A_1$ and $A_2$ from one ruling and two distinct lines $B_1$ and $B_2$ from the other ruling. Set $D_4:=A_1\cup A_2\cup B_1\cup B_2$. We choose two more lines, one from each ruling, distinct from the lines in $D_4$, and blow up the four points where they meet $D_4$. The resulting surface is $\mathfrak{Y}_4$, and $\mathfrak{D}_4$ is the strict transform of $D_4$. Note that all choices made are equivalent under the automorphism group of $\P^1\times \P^1$.
\end{example}

\begin{figure}[!ht]%
\begin{tikzpicture}

\node[vertex] (A1) at (-3,0.5) {};
\node[vertex] (B1) at (-2,0.5) {};
\node[vertex] (C1) at (-1,0.5) {};
\node[vertex] (D1) at (0,0.5) {};
\node[vertex] (E1) at (1,0.5) {};
\node[vertex] (F1) at (2,0.5) {};
\node[vertex] (G1) at (3,0.5) {};
\node[vertex] (H1) at (4,0.5) {};
\node[special] (I1) at (5,0.5) {};
\node[vertex] (J1) at (-1,1.5) {};

\draw (A1) -- (B1) -- (C1) -- (D1) -- (E1) -- (F1) -- (G1) -- (H1) -- (I1)
(C1) -- (J1);

\node[above, yshift=2mm, xshift=-1.3mm, font=\small] at (H1) {$-1$};

\end{tikzpicture}
\caption{The augmented curve structure  of the surface $\gothY_1$.}%
\label{figure gothy1 curve structure}%
\end{figure}
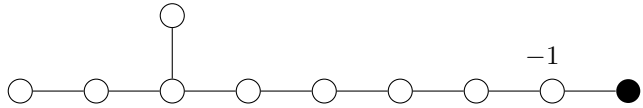

\begin{figure}%
\begin{tikzpicture}

\node[special] (A) at (0,0) {};
\node[vertex] (B) at (1,0) {};
\node[vertex] (C) at (2,0) {};
\node[vertex] (D) at (3,0) {};
\node[vertex] (E) at (4,0) {};
\node[vertex] (F) at (5,0) {};
\node[vertex] (G) at (6,0) {};
\node[vertex] (H) at (7,0) {};
\node[special] (I) at (8,0) {};
\node[vertex] (J) at (4,1) {};

\draw (A) -- (B);
\draw (B) -- (C);
\draw (C) -- (D);
\draw (D) -- (E);
\draw (E) -- (F);
\draw (F) -- (G);
\draw (G) -- (H);
\draw (H) -- (I);
\draw (E) -- (J);

\node[above, yshift=2mm, xshift=-1.3mm, font=\small] at (B) {$-1$};
\node[above, yshift=2mm, xshift=-1.3mm, font=\small] at (H) {$-1$};

\end{tikzpicture}
\caption{The augmented curve structure of the surface $\gothY_2$.}%
\label{figure gothy2 curve structure}%
\end{figure}
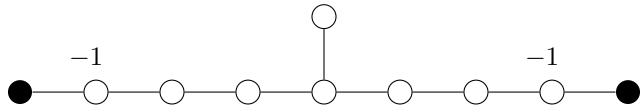

\begin{figure}%
\begin{tikzpicture}

\node[special] (A2) at (0,-1) {};
\node[vertex] (B2) at (1,-1) {};
\node[vertex] (C2) at (2,-1) {};
\node[vertex] (D2) at (3,-1) {};
\node[special] (E2) at (4,-1) {};
\node[special] (A3) at (0,-2) {};
\node[vertex] (B3) at (1,-2) {};
\node[vertex] (C3) at (2,-2) {};
\node[vertex] (D3) at (3,-2) {};
\node[special] (E3) at (4,-2) {};

\draw (A2) -- (B2);
\draw (B2) -- (C2);
\draw (C2) -- (D2);
\draw (C2) -- (C3);
\draw (D2) -- (E2);
\draw (A3) -- (B3);
\draw (B3) -- (C3);
\draw (C3) -- (D3);
\draw (D3) -- (E3);

\node[above, yshift=2mm, xshift=-1.3mm, font=\small] at (B2) {$-1$};
\node[above, yshift=2mm, xshift=-1.3mm, font=\small] at (D2) {$-1$};
\node[below, yshift=-2.5mm, xshift=-1.3mm, font=\small] at (B3) {$-1$};
\node[below, yshift=-2.5mm, xshift=-1.3mm, font=\small] at (D3) {$-1$};

\end{tikzpicture}

\caption{The augmented curve structure of the surface $\gothY_4$.}%
\label{figure gothy4 curve structure}%
\end{figure}
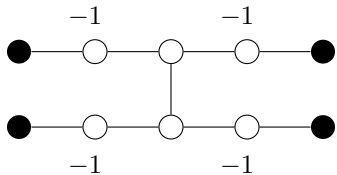

The surfaces $\gothY_1$, $\gothY_2$, and $\gothY_4$ show up in the construction of the models in $(-1)$-form, $Y_P$ and $Y_T$, as we will explain next.

\begin{example}\label{example yp}
The surface $Y_P$ has three identical irreducible components $Y_1=Y_2=Y_3= \gothY_2$, whose augmented curve structure is shown in Figure~\ref{figure gothy2 curve structure}. The black vertices represent the components of the anticanonical cycle and the white vertices have self-intersection $(-2)$ unless otherwise indicated. Call $D_{ij}$, $D_{ik}$ with $\{i,j,k\}=\{1,2,3\}$ the boundary components of $Y_i$. We 
denote the singular points of the anticanonical cycle on~$Y_i$ by $Q_{i1}$ and  $Q_{i2}$.  We note that there is an automorphism of $Y_i$ which interchanges the components $D_{ij}$ and $D_{ik}$, while leaving the points $Q_{i1}$ and  $Q_{i2}$ fixed (as a set). 
Similarly, there is an automorphism leaving
the components $D_{ij}$ and $D_{ik}$ fixed (as sets), while swapping the points $Q_{i1}$ and  $Q_{i2}$. 
The existence of these automorphisms can, for example, be deduced from the description given in Example~\ref{example geometric description gothy2}. In terms of the ruled surface $\mathbb F_2$
this corresponds to an automorphism interchanging $S_1$ and $S_2$, but fixing their intersection points and the ruling (as a set), or,
vice versa, leaving  $S_1$, $S_2$ and the ruling fixed (as sets)  and interchanging the two points of intersection. In this way we obtain a group $\mu_2 \times \mu_2$ of automorphisms acting on the surface $\mathfrak Y_2$.

Further, let $P_{ij}\in D_{ij}$ denote the points corresponding to the intersections of the white $(-1)$-vertices in the curve structure with their adjacent black vertices. 
These points will be called interior special points, see Definition~\ref{definition interior special points}. 

We can now describe how to construct the surface $Y_P$. We start with an ordered triple $Y_i$, where $i=1,2,3$, of components isomorphic to $\mathfrak Y_2$. 
For each component $Y_i$, we fix a labelling $D_{i,j}$ for $j \in \{1,2,3\} \setminus\{i\}$  of the components of the anticanonical cycle and the singular points 
$Q_{i,k}$ for $k=1,2$ of this cycle. 
The surface $Y_P$ is now obtained by gluing the $Y_i$ along certain isomorphisms of the components of the anticanonical cycles.
Since these components are smooth rational curves, giving isomorphisms between them is tantamount to specifying ordered $3$-tuples of points on them that are to be identified under the isomorphism. 
More precisely, we choose isomorphisms $D_{12} \cong D_{21}$, $D_{23} \cong D_{32}$ and $D_{31} \cong D_{13}$ which are specified by mapping $Q_{i,k}$ to $Q_{i+1,k}$ for $i=1,2,3$ and 
$j=1,2$ as well as sending interior special points to interior special points. 
Using the symmetry groups $\mu_2 \times \mu_2$ of the surfaces $Y_i$, which we have described above, shows that the resulting surface $Y_P$ is independent of any choices that have been made for the labelling of the 
components $D_{ij}$ and the singular points $Q_{ik}$. 
\end{example}

\begin{example}\label{example yt}
The surface $Y_T$ has irreducible components $Y_1=Y_2=\gothY_1$ and $Y_3$ a nonnormal surface with normalization $\gothY_4$. We call $Y_3$ the \emph{special component} and $Y_1$, $Y_2$ the \emph{nonspecial components}. The curve structures of $Y_1=Y_2=\gothY_1$ and \mbox{$(Y_3)^\nu=\gothY_4$} are shown in Figure~\ref{figure gothy1 curve structure} and Figure~\ref{figure gothy4 curve structure}, respectively.
To obtain the surface $Y_T$, one glues the components as follows. First, one glues two opposite boundary components to one another, e.g. those represented by the black vertices in the bottom row of Figure~\ref{figure gothy4 curve structure}, such that the intersection points with the remaining boundary components are matched and also the intersection points with the curves corresponding to unique white vertices adjacent to the glued boundary components are matched. This specifies the gluing uniquely. After this first gluing, the remaining boundary components, i.e. those represented by the black vertices in the top row of Figure~\ref{figure gothy4 curve structure}, have become nodal rational curves. These are then glued to the boundary curves of the nonspecial components, again matching singular points and intersection points with neighboring curves in the curve structure.
\end{example}

We extend the above terminology to all combinatorial K3 surfaces of class $T$: The unique\footnote{The existence and uniqueness can directly be read off from the triangulation $T$.} irreducible component with four boundary components in the anticanonical cycle is called \emph{special} while the other components are called \emph{nonspecial}. By construction, a type I flop preserves special (and hence also nonspecial) components.

\subsection{Curve structures}\label{section curve structures}

A curve structure is a combinatorial object associated to a combinatorial K3 of degree $2$. We have met first instances in Example~\ref{example gothy}, here we define them for arbitrary components of surfaces in $\CK2$. Recall the following result from \cite[Corollary~6.32]{HL22}.
\begin{theorem}\label{theorem type i flops to minus one model}
Let $Y$ be a combinatorial K3 surface of degree $2$. Then there is a sequence
\begin{equation}\label{eq sequence of type i flops}
Y_G^\nu \ratl \ldots \ratl Y^\nu
\end{equation}
of type I flops where $G\in\{P,T\}$. In particular, $\CK2(G)$ is the subclass of $\CK2$ of surfaces that can be obtained from $Y_G$ by a sequence of type I flops.
\end{theorem}

In \cite{HL22}, the sequence \eqref{eq sequence of type i flops} was between the surfaces themselves and not the normalizations. With our definition, however, of a type I flop, giving a flop on the normalizations is equivalent to giving a flop on the glued surfaces.

Recall from Example~\ref{example gothy} that $C(\gothY_i)$ is the set of square-negative non-anticanonical curves on $\gothY_i$. 
Restricting a birational morphism as in \eqref{eq sequence of type i flops} to the surface $\gothY_i$, one obtains a set $C(Y_i)$ for each irreducible component $Y_i$ of $Y$ by adding the exceptional curve in case of a blow-up and removing it in case of a blow-down, see \cite[4.2]{HL22}. It is shown in \cite[Lemma~4.9]{HL22} that the set $C(Y_i)$ is independent of the sequence of type~I flops $Y_G \ratl Y$. The definition of the curve structure of $Y_i$ proceeds now analogously to the one of $\gothY_j$ in Example~\ref{example gothy}.

\begin{definition}\label{definition curve structure} 
Let $Y_1$ be an irreducible component of a surface $Y \in \CK2$. We define the \emph{curve structure} $\Gamma_{Y_1}$ of $Y_1$ to be the dual graph of $C(Y_1)$ with vertices labeled by the self-intersection numbers of the corresponding curves. Let $Y_1^\nu \to Y_1$ be the normalization, and let $D$ be an integral curve in the anticanonical cycle of $Y^\nu$. 
For each such $D$, we append a vertex $d$ to $\Gamma_{Y_1}$ and for each $v\in \Gamma_{Y_1}$ we add an edge joining $d$ and $v$ if $D.C_v  \neq 0$ where $C_v$ is the curve corresponding to $v$. The labeled graph $\Gamma_{Y_1}^a$ thus obtained is called the \emph{augmented} curve structure of $Y_1$. We define the curve structure respectively the augmented curve structure of $Y$ as the disjoint union of the curve structures respectively augmented curve structures of its components.
\end{definition}

\begin{remark}
We will usually interpret $C(Y)$ as a set of curves on the normalization $Y^\nu$. In particular, all intersection numbers will be computed on $Y^\nu$.
\end{remark}

We use the word 'vertex' sometimes as a synonym for the curve it represents. The black vertices are referred to as \emph{anticanonical vertices}. Note that an anticanonical vertex is however not the same as an anticanonical curve. The former is just an irreducible component of the anticanonical cycle.

We come to the most important property of curve structures, proven in \cite{HL22}.

\begin{theorem}\label{theorem curve structures determine surface}
A degree $2$ projective combinatorial K3 surface is uniquely determined by its (augmented) curve structure.
\end{theorem}
\begin{proof}
For a curve structure of class $P$, this holds by \cite[Proposition~7.5]{HL22}.
For a curve structure of class $T$, this is \cite[Proposition~7.6]{HL22}. Since the augmented curve structure can distinguish the dual complex of the surface, the claim follows.
\end{proof}

\section{Interior special points}\label{section interior special points}

We will need the notion of \emph{special points} in order to specify the gluing condition of type II flops. The key idea is that there is a unique way of gluing two smooth rational curves when the gluing identifies three marked points on each. These are the special points. We give an explicit description of the special points for degree $2$ combinatorial K3 surfaces.

Then we discuss certain distinguished rulings on components of combinatorial K3 surfaces of degree $2$. This will help understand the effect of a type II flop on the curve structure of a combinatorial K3 surface of degree $2$. We also use it to describe how interior special points emerge in a type II flop. More precisely, it is needed to give an explicit description of a type II flop that does not resort to the Carlson map, cf. Remark~\ref{remark carlson map}.

\begin{lemma}\label{lemma interior special points}
Let $Y$ be a combinatorial K3 surface of degree $2$ and let $\nu:Y^\nu \to Y$ be its normalization. If $D \subset Y^\nu$ is an irreducible curve in the preimage of the double locus, then there is a unique point $p\in D$ such that $\nu(p)$ is not a triple point and every curve in the curve structure of $Y$ either meets $D$ in $p$ or not at all.
\end{lemma}
\begin{proof}
This holds for the surfaces $\gothY_i$, $i=1,2,4$, from Example~\ref{example gothy} and is preserved under type I flops. Then the claim follows from Theorem~\ref{theorem type i flops to minus one model}.
\end{proof}

\begin{definition}\label{definition interior special points}
We put ourselves in the setup of the previous lemma. We call the point $p$ the \emph{interior special point} of the curve $D$. A point in $D$ is \emph{special} if it is either interior special or maps to one of the triple points under $\nu$.
\end{definition}

Abstractly, the maximality condition determines the special interior points in all degrees by the study of the Carlson map. We were however unable to describe the special interior points explicitly in this approach, which is why we restricted to degree~$2$ and gave the explicit description above.

\subsection{Rulings on combinatorial K3 surfaces of class \texorpdfstring{$T$}{T}}\label{section rulings class t}

We describe certain distinguished linear systems on the nonspecial components of combinatorial K3 surfaces of degree $2$ and class $T$. In most cases, these are rulings and this is what we start with. 

Let us consider the surface $\gothY_1$ and denote by $v_1,\ldots,v_9$ the basis of $\Pic(\gothY_1)$ given by the elements of the curve structure, numbered according to Figure~\ref{figure isotropic vector gothy1} (where we wrote $i$ instead of $v_i$ for the sake of readability). Then $\ell:=(0,1,2,2,2,2,2,2,1)$ is a vector of square zero and we choose a line bundle $L$ on $\gothY_1$ representing $\ell$. Let us denote by $C_i\subset \gothY_1$ the curve representing the vertex $v_i$ for $i=1,\ldots, 9$.

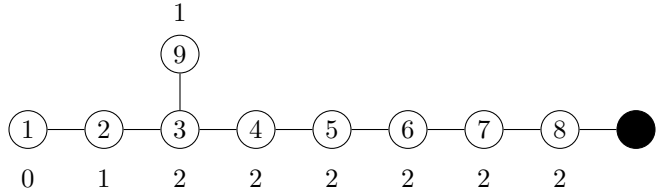
\begin{figure}[!h]%
\begin{tikzpicture}[font=\small,
  bigvertex/.style={
    circle,
    draw=black,
    fill=white,
    minimum size=5mm,  
    inner sep=0pt
  },
  biganticanonical/.style={
    circle,
    draw=black,
    fill=black,
    minimum size=5mm,
    inner sep=0pt
  }
]

\node[bigvertex] (A1) at (-3,0.5) {1};
\node[bigvertex] (B1) at (-2,0.5) {2};
\node[bigvertex] (C1) at (-1,0.5) {3};
\node[bigvertex] (D1) at (0,0.5) {4};
\node[bigvertex] (E1) at (1,0.5) {5};
\node[bigvertex] (F1) at (2,0.5) {6};
\node[bigvertex] (G1) at (3,0.5) {7};
\node[bigvertex] (H1) at (4,0.5) {8};
\node[biganticanonical] (I1) at (5,0.5) {}; 
\node[bigvertex] (J1) at (-1,1.5) {9}; 

\draw (A1) -- (B1) -- (C1) -- (D1) -- (E1) -- (F1) -- (G1) -- (H1) -- (I1)
      (C1) -- (J1);

\node[below, yshift=-4mm, xshift=0mm] at (A1) {$0$};
\node[below, yshift=-4mm, xshift=0mm] at (B1) {$1$};
\node[below, yshift=-4mm, xshift=0mm] at (C1) {$2$};
\node[below, yshift=-4mm, xshift=0mm] at (D1) {$2$};
\node[below, yshift=-4mm, xshift=0mm] at (E1) {$2$};
\node[below, yshift=-4mm, xshift=0mm] at (F1) {$2$};
\node[below, yshift=-4mm, xshift=0mm] at (G1) {$2$};
\node[below, yshift=-4mm, xshift=0mm] at (H1) {$2$};

\node[above, yshift=3mm, xshift=0mm] at (J1) {$1$};

\end{tikzpicture}

\caption{The isotropic vector $\ell$ on $\gothY_1$. The vertices are numbered as indicated, the coefficients of $\ell$ at each vertex are below the vertex for $v_1,\ldots, v_8$ respectively above for $v_9$.}%
\label{figure isotropic vector gothy1}%
\end{figure}

\begin{lemma}\label{lemma zero curve on gothy1}
The line bundle $L$ is nef, its linear system is one-dimensional and gives rise to a fibration $\gothY_1 \to \P^1$. The anticanonical curve $D$ is a $2$-section of the fibration, the curve $C_1$ is a section, and the union of all other curves $C_i$ in the curve structure is the support of the unique singular fiber.
\end{lemma}
\begin{proof}
This can be read off from the description in Example~\ref{example geometric description gothy1}. The sought-for fibration is given by pulling back the linear system of lines through the chosen inflection point of the nodal cubic in $\P^2$. Indeed, $C_1$ is the exceptional divisor of the first blow-up (and thus a section) and $\ell$ corresponds to the unique singular fiber which is the pullback of the flex line.
\end{proof}

\begin{remark}\label{remark unique fibration}
Thanks to the explicit description of the cone of curves, see Proposition~\ref{proposition cone generators} below, one can verify with a computer program that $\ell$ is in fact the only primitive isotropic vector in the nef cone of $\gothY_1$. This is however not needed in what follows.
\end{remark}

\begin{corollary}\label{corollary square zero curve class t}
Let $Y\in\CK2$ be a surface of class $T$ and let $Y_T \ratl Y$ be a composition of type I flops. Let $S \subset Y$ be a nonspecial component and consider the rational map \mbox{$\phi:\gothY_1 \ratl S$} 
given by 
the restriction of the above composition. If $S$ has Picard rank $\rho(S) \geq 2$,  there is a unique fibration $S\to \P^1$ that corresponds to the one from 
Lemma~\ref{lemma zero curve on gothy1} under~$\phi$. This fibration has a unique singular fiber if $\rho(S)\geq 3$ and is smooth otherwise.
\end{corollary}
\begin{proof}
From the description in Example~\ref{example geometric description gothy1}, using the assumption $\rho(S)\geq 2$, we see that $S$ is obtained from $\gothY_1$ by blow-ups or blow-downs in the unique singular fiber of the fibration $\gothY_1 \to \P^1$. This proves the existence of the fibration, which by construction has at most one singular fiber. The claim follows.
\end{proof}

\begin{corollary}\label{corollary unique square zero curve class t}
Let $S$ be a nonspecial component of a combinatorial K3 surface of class $T$ and let $p\in S$ be the node of the anticanonical curve. If $\rho(S) \geq 2$, there is a unique irreducible (square zero) curve in the linear system of $L$ passing through $p$.
\end{corollary}
\begin{proof}
The linear system $|L|$ is base point free by Corollary~\ref{corollary square zero curve class t}. Thus, the point $p$ lies in a unique fiber of the corresponding morphism $S \to \P^1$. It remains to show that this fiber is irreducible. 
This follows from the description in Example~\ref{example geometric description gothy1}, since the only (possibly) singular fiber is the preimage of the flex line while the fiber through~$p$ is the strict transform of the line in $\P^2$ through the node and the chosen inflection point.
\end{proof}

\begin{example}\label{example class t picard rank one}
By the enumeration of models, the only possibility for a surface $S$ as in  Corollary~\ref{corollary square zero curve class t}  with Picard number $1$ is the following. With the curves in the curve structure numbered as in Figure~\ref{figure isotropic vector gothy1}, contract successively the curves $C_8$, $C_7$, $C_6$, $C_5$, $C_4$, $C_3$, $C_2$, $C_1$ (respectively their strict transforms). In this case, $S\isom \P^2$ and the anticanonical curve is a nodal cubic curve. The curve structure is given by one vertex, which is represented by a line. The final contraction, call it $\wt S \to S$, is shown in Figure~\ref{figure final contraction class t} at the level of curve structures. Lines passing through the interior special point on $S$ form a one-dimensional linear system, and there is a unique such line $F$ passing through the node of the anticanonical curve. The figure illustrates that this linear system is exactly the image of the linear system of the $0$-curve on $S_1$ and the curve from Corollary~\ref{corollary unique square zero curve class t} is exactly the strict transform $\wt F$ of $F$.

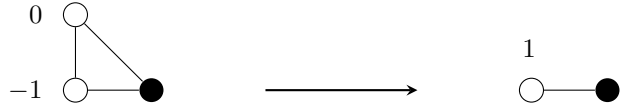
\begin{figure}[!h]%
\begin{tikzpicture}[>=stealth]

\node[vertex] (A) at (-2,1.5) {};
\node[vertex] (B) at (-2,0.5) {};
\node[special] (C) at (-1,0.5) {};

\begin{scope}[xshift=3cm]
\node[vertex] (D) at (1,0.5) {};
\node[special] (E) at (2,0.5) {};

\end{scope}

\draw (A) -- (B) -- (C) -- (A)
(D) -- (E);

\node[left, yshift=0mm, xshift=-3mm, font=\small] at (A) {$0$};
\node[left, yshift=0mm, xshift=-3mm, font=\small] at (B) {$-1$};

\node[above, yshift=3mm, xshift=-0.3mm, font=\small] at (D) {$1$};

\draw[->, thick] (0.5,0.5) -- (2.5,0.5);

\end{tikzpicture}
\caption{The contraction of $C_1$.}%
\label{figure final contraction class t}%
\end{figure}

\end{example}

\begin{definition}\label{def:specialpointsT}
Let $S$ be a nonspecial component of a combinatorial K3 surface of class $T$ and let $p\in S$ be the node of the anticanonical curve $D$. Let $\pi:\wt S \to S$ be the blow-up in the node and 
let $E:=\pi^{-1}(p)$ be the exceptional curve. Then we define an interior special point $\wt p$ of $E$ as follows. If $\rho(S) \geq 2$, there is a unique irreducible square zero curve $F$ in the 
linear system $|L|$ passing through $p$ by Corollary~\ref{corollary unique square zero curve class t}. We take $\wt p$ to be the unique intersection point $E \cap \wt F$ where $\wt F$ is the 
strict transform of $F$. Note that $\wt p$ is indeed interior for if $\wt F$ met the singular points of $\pi^{-1}(D)$, the curve $F$ would have to meet some branch of the node with multiplicity $\geq 2$, hence $F.D \geq 3$ contradicting the adjunction formula.
If $\rho(S)=1$, we are in the situation of Example~\ref{example class t picard rank one}, in particular, $S\isom \P^2$. 
Here, we take $\wt p$ to be the unique intersection point $E \cap \wt{F}$ where $\wt{F}$ is the strict transform of the line $F$ joining the interior special point of $S$ with the node of the anticanonical curve.
\end{definition}

Note that if $\wt S$ is an irreducible component of a combinatorial K3 surface (necessarily of class $P$), then $\pi^*D = \wt D + E$ is the anticanonical cycle where $\wt D$ is the strict transform of the anticanonical curve $D$. Then the point $\wt p$ is indeed an interior special point according to our previous definition (Definition~\ref{definition interior special points}): it is the intersection of the $(-1)$-curve $\wt F$ with $E$ and $\wt F$ is in the curve structure by Corollary~\ref{corollary negative curves in curve structure}.

\subsection{Rulings on combinatorial K3 surfaces of class \texorpdfstring{$P$}{P}}\label{section rulings class p}

We consider the surface $\gothY_2$ and denote by $v_1,\ldots,v_8$ the basis of $\Pic(\gothY_2)$ given by the elements of the curve structure, numbered according to Figure~\ref{figure isotropic vector gothy2} (where again we wrote $i$ instead of $v_i$ for the sake of readability). Then $\ell:=(1,1,1,1,1,1,1,0)$ is a vector of square zero and we choose a line bundle $L$ on $\gothY_2$ representing $\ell$. Let us denote by $C_i\subset \gothY_2$ the curve representing the vertex $v_i$ for $i=1,\ldots, 8$.

\begin{figure}[!h]%
\begin{tikzpicture}[font=\small,
  bigvertex/.style={
    circle,
    draw=black,
    fill=white,
    minimum size=5mm,  
    inner sep=0pt
  },
  biganticanonical/.style={
    circle,
    draw=black,
    fill=black,
    minimum size=5mm,
    inner sep=0pt
  }
]

\node[biganticanonical] (Z1) at (-4,0.5) {}; 
\node[bigvertex] (A1) at (-3,0.5) {1};
\node[bigvertex] (B1) at (-2,0.5) {2};
\node[bigvertex] (C1) at (-1,0.5) {3};
\node[bigvertex] (D1) at (0,0.5) {4};
\node[bigvertex] (E1) at (1,0.5) {5};
\node[bigvertex] (F1) at (2,0.5) {6};
\node[bigvertex] (G1) at (3,0.5) {7};
\node[biganticanonical] (I1) at (4,0.5) {}; 
\node[bigvertex] (J1) at (0,1.5) {8};  

\draw (Z1) -- (A1) -- (B1) -- (C1) -- (D1) -- (E1) -- (F1) -- (G1) -- (I1)
      (D1) -- (J1);

\node[below, yshift=-4mm, xshift=0mm] at (A1) {$1$};
\node[below, yshift=-4mm, xshift=0mm] at (B1) {$1$};
\node[below, yshift=-4mm, xshift=0mm] at (C1) {$1$};
\node[below, yshift=-4mm, xshift=0mm] at (D1) {$1$};
\node[below, yshift=-4mm, xshift=0mm] at (E1) {$1$};
\node[below, yshift=-4mm, xshift=0mm] at (F1) {$1$};
\node[below, yshift=-4mm, xshift=0mm] at (G1) {$1$};

\node[above, yshift=3mm, xshift=0mm] at (J1) {$0$};

\end{tikzpicture}

\caption{The isotropic vector $\ell$ on $\gothY_2$. The vertices are numbered as indicated, the coefficients of $\ell$ at each vertex are below the vertex for $v_1,\ldots, v_7$ respectively above for $v_8$.}%
\label{figure isotropic vector gothy2}%
\end{figure}
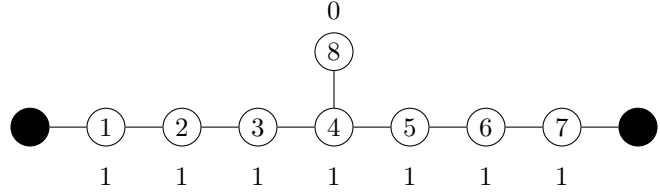

\begin{lemma}\label{lemma zero curve on gothy2}
The line bundle $L$ is nef, its linear system is one-dimensional and gives rise to a fibration $\gothY_2 \to \P^1$. The components $D_1, D_2$ of the anticanonical cycle and the curve $C_8$ are sections of the fibration, the union of all other curves $C_i$ in the curve structure is the support of the unique singular fiber.
\end{lemma}
\begin{proof}
Similarly to Lemma~\ref{lemma zero curve on gothy1}. Also, for $\gothY_2$, there are irreducible square zero curves in the linear system of $\ell$ that one can see in the construction of $\gothY_2$.

Alternatively, we can successively blow down the curves $C_1, C_2, C_3, C_7, C_6, C_5$. Then we are left with a smooth surface of Picard rank $2$. The images of $C_4, C_8$ (which we denote by the same symbol) are curves of square $0$ respectively $-2$ and $|C_4|$ gives a ruling with only smooth fibers. Since $C_8$ is still a section of square $-2$, we infer from the classification of surfaces that the contracted surface must be the second Hirzebruch surface $\F_2$. We obtain the sought-for fibration as the composite map $Y \rightarrow \F_2\rightarrow \P^1$.
\end{proof}

\begin{remark}
Unlike for $\gothY_1$, where the fibration was unique by Remark~\ref{remark unique fibration}, for $\gothY_2$ there are three linearly independent square zero classes on the boundary of the nef cone. In the above basis, the others are given by 
\[
\ell'=(0, 0, 1, 2, 2, 2, 2, 1) \quad \textrm{ and } \quad \ell''=(2, 2, 2, 2, 1, 0, 0, 1).
\]
For the first one, $C_2$ is a section and there are two singular fibers with support $D_1 \cup C_1$ respectively $\cup_{i\neq 1,2}C_i$, for the second one, $C_6$ is a section and there are two singular fibers with support $\cup_{i\neq 6,7} C_i$ respectively $C_7 \cup D_2$. Here, $D_1$ is the component of the anticanonical cycle adjacent to $C_1$ and $D_2$ the one adjacent to $C_7$. Thus, the class $\ell$ depicted in Figure~\ref{figure isotropic vector gothy2} defines the only fibration on $\gothY_2$ such that the components of the anticanonical cycle are all horizontal.
\end{remark}

We continue to use the notation from the beginning of Section~\ref{section rulings class p}.

\begin{corollary}\label{corollary square zero curve class p}
Let $Y\in\CK2$ be a surface of class $P$ and let $Y_P \ratl Y$ be a composition of type I flops. Let $S \subset Y$ be an irreducible component and consider the rational map $\phi:\gothY_2 \ratl S$ given by 
the restriction of the above composition. If $\phi$ is an isomorphism at the generic point of the curve $C_8$, then there is a unique fibration $S\to \P^1$ that corresponds to the one from 
Lemma~\ref{lemma zero curve on gothy2} under~$\phi$. This fibration has at most one singular fiber.
\end{corollary}
\begin{proof}
Since $\phi$ is regular at the generic point of $C_8$, it is a composition of blow-ups and blow-downs in the unique singular fiber of the morphism $\gothY_2\to \P^1$ given by the linear system of $L$, see Lemma~\ref{lemma zero curve on gothy2}. Thus, it preserves the fibration and does not alter the smooth fibers. The claim follows. 
\end{proof}

\begin{figure}[!h]%
\begin{tikzpicture}

\node[special] (A) at (0,0) {};
\node[vertex] (B) at (1,0) {};
\node[vertex] (C) at (2,0) {};
\node[vertex] (E) at (4,0) {};
\node[special] (I) at (5,0) {};

\node (dots) at (3,0) {$\cdots$};

\draw (A) -- (B) -- (C) -- (dots) -- (E) -- (I);

\node[above, yshift=2mm, xshift=-1.3mm, font=\small] at (B) {$\phantom{-}0$};
\node[above, yshift=2mm, xshift=-1.3mm, font=\small] at (E) {$-1$};

\end{tikzpicture}
\caption{A general very degenerate curve structure.}%
\label{figure very degenerate curve structure}%
\end{figure}
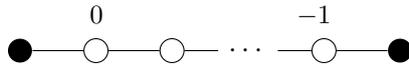

\begin{remark}
Let $\phi:\gothY_2 \ratl S$ be an iteration of blow-ups and blow-downs in interior special points. Then it is straightforward to show that $\phi$ contracts $C_8$ if and only if the curve structure on $S$ is very degenerate in the sense of \cite[Definition~4.18]{HL22}. Such curve structures are shown in Figure~\ref{figure very degenerate curve structure}. By Corollary~\ref{corollary square zero curve class p}, if $S$ is regular, that is, the opposite of very degenerate, the fibration from Lemma~\ref{lemma zero curve on gothy2} remains regular. It is easy to see that this is if and only if: as soon as the section is contracted, all fibers intersect and the fibration cannot remain regular. Thus, we obtain a geometric characterization of being regular or very degenerate as preserving or destroying a certain fibration.
\end{remark}

Similarly as in Corollary~\ref{corollary unique square zero curve class t}, one shows the following.

\begin{corollary}\label{corollary unique square zero curve class p}
Let $S$ be a component of a combinatorial K3 surface of class $P$ and suppose that $C_8$ is not contracted under the rational map $\phi:\gothY_2 \ratl S$. Then, for each triple point $p$ on $S$, there is a unique irreducible (square zero) curve in the linear system of $L$ passing through $p$.\qed
\end{corollary}

\begin{definition}\label{def:specialpointsP}
Let $S$ be a component of a combinatorial K3 surface of class $P$ and denote by $D_1$, $D_2$ the components of the anticanonical cycle of $S$. Then $D_1\cap D_2$ consists of two points, call them $p_1$, $p_2$. Let $\pi:\wt S \to S$ be the blow-up in $\{p_1,p_2\}$ and denote $E_i=\pi^{-1}(p_i)$ for $i=1,2$ the exceptional curves. 

We define an interior special points $\wt p_i$ of $E_i$ as follows. If the curve $C_8$ is not contracted on $S$, then for each $i=1,2$ there is a unique irreducible square zero curve $F_i$ in the linear system $|L|$ passing through $p_i$ by Corollary~\ref{corollary unique square zero curve class p}. We take $\wt p_i$ to be the unique intersection point $E_i \cap \wt F_i$ where $\wt F_i$ is the strict transform of $F_i$. As in Definition~\ref{def:specialpointsT}, one sees that this point is \emph{interior}. If $C_8$ is contracted on $S$, we write $S$ as a birational contraction $S'\to S$ of a surface $S'$ where $C_8$ is not contracted. We take the linear system there and take its image on $S$. Then, for each $i$ there is a unique irreducible curve $F_i$ in the linear system joining $p_i$ to the special interior point of $D_i$. We take $\wt p_i$ to be the unique intersection point $E_i \cap \wt{F}_i$ where $\wt{F}_i$ is the strict transform of $F_i$.
\end{definition}

Note that if $\wt S$ is an irreducible component of a combinatorial K3 surface (necessarily the special component of a surface of class $T$), then $\pi^*(D_1+D_2) = \wt D_1 + \wt D_2 + E_1 + E_2$ is the anticanonical cycle where $\wt D_i$ is the strict transform of $D_i$. Then the point $\wt p_i$ is indeed an interior special point according to Definition~\ref{definition interior special points}: it is the intersection of the $(-1)$-curve $\wt F_i$ with $E_i$ and $\wt F_i$ is in the curve structure by Corollary~\ref{corollary negative curves in curve structure}.

\section{Type II flops}\label{section type ii flops}

We now turn to type II flops of degree $2$ combinatorial K3 surfaces. First, we will describe the flops themselves in Section~\ref{section type ii in degree two}. Then, we will discuss their effect on the curve structures in Section~\ref{section type ii flops on curve structures}.

\subsection{Type II flops of combinatorial K3 surfaces of degree \texorpdfstring{$2$}{2}}\label{section type ii in degree two}

Let us consider a combinatorial $K3$ surface and a smooth component $C=D_{ij}$ of the double locus of $Y$. Here, we adopt the notation of Section~\ref{section type i flops}. The preimage of $C$ under the normalization map $\nu: Y^{\nu} \to Y$ contains exactly two one-dimensional irreducible components and we assume that these are $(-1)$-curves in the respective normalizations of the components.  The case $i=j$ is allowed, in which case $Y$ is a class $T$ surface and the preimage of $C=D_{ii}$ consists of two smooth disjoint $(-1)$-curves $D^1_{ii} \cup D^2_{ii}$. Otherwise, $Y$ is necessarily a class $P$ surface.  

\begin{construction}\label{construction type ii flop pt}
We shall first treat the case $i\neq j$, which means that $Y$ is of class~$P$. 
In this case, the curve $C$ intersects the third component $Y_k$ in two points $P^1_k$ and $P^2_k$. In this case we proceed as follows. We contract the $(-1)$ 
curves $D_{ij} \subset Y^{\nu}_i=Y_i$  and $D_{ji} \subset Y^{\nu}_j=Y_j$ to obtain smooth surfaces $Y'^{\nu}_i$ and $Y'^{\nu}_j$. The anticanonical cycles of these surfaces are now nodal rational curves.   
The component  $Y_k=Y^{\nu}_k$ is also smooth and we blow it up in $P^1_k$ and $P^2_k$ to obtain a smooth surface $Y'^{\nu}_k$ with exceptional curves $E_1$ and $E_2$. The anticanonical cycle
of this surface now consists of the cycle of four smooth components $D'^{\nu}_{ki} + E_1 + D'^{\nu}_{kj}  + E_2$ where $D'^{\nu}_{ki}$ and $D'^{\nu}_{kj}$ are the strict transforms of $D_{ki}$ and $D_{kj}$.
The curves $D'^{\nu}_{ki}$ and $D'^{\nu}_{kj}$  have three special points, namely the singular points of the anticanonical cycle and the preimages of the interior special points. 
These are matched naturally to the three special points on 
$D'^{\nu}_{ik}$ and $D'^{\nu}_{jk}$. We now glue the two disjoint $(-1)$-curves $E_1$ and $E_2$ using the two singular points of the singular locus of the anticanonical cycle and the special interior 
points on $E_1$ and $E_2$ given by Definition~\ref{def:specialpointsP}. 
In this way we obtain a singular surface $Y'_k$ and the images $D'_{ki}$ and $D'_{kj}$ of $D'^{\nu}_{ki}$ and $D'^{\nu}_{kj}$ become nodal curves. 
We now glue $Y'_i$ and  $Y'_k$ along the nodal curves $D'_{ik}$ and $D'_{ki}$, thereby using the three special points on the normalizations of these curves  (which become a node and an interior special point).
 Similarly, we glue  $Y'_j$ and  $Y'_k$. The result is a combinatorial K3 surface $Y'$. Our matching of the special points ensures that the Carlson map is trivial and hence that $Y'$ is maximal. 
\end{construction}

We depict this construction and the next in Figure \ref{figure type ii flop}.

\begin{construction}\label{construction type ii flop tp}
We shall now consider the case $i=j$. Then we are in class $T$ with $Y_i$ being the special singular component. On the normalization $Y^{\nu}_i$ the preimage of $C=D_{ii}$ is a disjoint union  
$D^1_{ii} + D^2_{ii}$ of $(-1)$-curves. The anticanonical  cycle is $D^{\nu}_{ij} + D^1_{ii} + D^{\nu}_{ik} + D^2_{ii}$. We blow down $D^1_{ii}$ and $D^2_{ii}$ and obtain a smooth 
surface $Y'_2$ with anticanonal cycle $D'_{ij} + D'_{ik}$. Note that we have three special points on each of these curves. The curve $C$ intersects the two smooth 
components $Y_j$ and $Y_k$ in points $P_j$ and $P_k$ respectively. Blowing these up gives smooth surfaces $Y'_j$ and $Y'_k$ whose anticanocial cycles have two components
$D'_{ji} + E_i$ and $D'_{ki} + E_k$ where $E_j$ and $E_k$ are the newly introduced exceptional curves. In this we we obtain the union $Y'^{\nu} = \cup_{\ell} Y'_{\ell}$. We want to glue these to 
a singular surface $Y'$. For this we identify $D'_{ij}$ and $D'_{ji}$ as well as $D'_{ik}$ and $D'_{ki}$ using the  special points we have on these curves. Finally, we glue $E_j$ and $E_k$. For 
this we use the singular points of the anticanonical cycle and the new interior special points given by Definition~\ref{def:specialpointsT}. Again, the surface $Y'$ has trivial 
Carlson homomorphism $c_{Y'}=1$ and $Y'$ is thus maximal.      
\end{construction}

\begin{definition} \label{def:type II flops final}
In the situation of Construction~\ref{construction type ii flop pt} or Construction~\ref{construction type ii flop tp}, we say that the combinatorial K3 surface $Y'$ {\em arises from $Y$ by the type II flop in the curve $C$} and we refer to the birational map $\phi: Y \dashrightarrow Y'$ as a 
{\em type II flop} of combinatorial K3 surfaces.   
\end{definition}  

\begin{remark}
Arguing as in the case of type I flops, one sees again that the type II flops described here are the restrictions of type II flops 
$\tilde \phi: \sY \dashrightarrow \sY'$ of maximal Kulikov models.
\end{remark}
\begin{remark}
Type II flops change the class of the degree $2$ combinatorial K3 surface, i.e. $Y'$ is of class $T$ (or $P$) if and only if $Y$ is of class $P$ (or $T$).
\end{remark}

\begin{figure}[!h]
\centering
\begin{tikzpicture}[scale=1]
\tikzset{mygreen/.style={green!66!black}}

\draw[] 
(-5,1.3)--(-4,0.3)
(-2,0.3)--(-1,1.3)
(-5,-1.3)--(-4,-0.3)
(-2,-0.3)--(-1,-1.3)
(-5,1)--(-4,0)--(-5,-1)
(-1,1)--(-2,0)--(-1,-1);

\draw[red] (-4,0.3)--(-2,0.3)
(-4,-0.3)--(-2,-0.3);
\node[above] at (-3,0.3) {\textcolor{red}{\footnotesize $D_{ji}$}};
\node[below] at (-3,-0.3) {\textcolor{red}{\footnotesize $D_{ij}$}};

\fill[mygreen] (-4,0) circle (2pt);
\node[right, xshift=1mm] at (-4,0) {\textcolor{green!66!black}{\footnotesize $P^\nu_k$}};
\fill[mygreen] (-2,0) circle (2pt);
\node[left, xshift=-1mm] at (-2,0) {\textcolor{green!66!black}{\footnotesize $P^\nu_\ell$}};

\node(A1) at (-3,1.7) {$Y_j^\nu$};
\node(B1) at (-3,-1.8) {$Y_i^\nu$};
\node(C1) at (-5.5,0) {$Y^{\nu}_k$};
\node(D1) at (-0.5,0) {$Y^{\nu}_\ell$};

\begin{scope}[xshift=2.5cm]

\draw[] (4,2)--(3,1)--(2,2) 
(4.3,-2)--(3.3,-1)
(3.3,1)--(4.3,2)
(1.7,-2)--(2.7,-1)
(2.7,1)--(1.7,2)
(4,-2)--(3,-1)--(2,-2);

\draw[red] (3.3,-1)--(3.3,1)
(2.7,-1)--(2.7,1);

\node[left] at (2.7,0) {\textcolor{red}{\footnotesize $D'_{k\ell}$}};
\node[right] at (3.3,0) {\textcolor{red}{\footnotesize $D'_{\ell k}$}};

\fill[mygreen] (3,-1) circle (2pt);
\node[above, yshift=1mm] at (3,-1) {\textcolor{green!66!black}{\footnotesize $P^{\prime\nu}_i$}};
\fill[mygreen] (3,1) circle (2pt);
\node[below, yshift=-1mm] at (3,1) {\textcolor{green!66!black}{\footnotesize $P^{\prime\nu}_j$}};

\node(A2) at (1.5,0) {$Y_k^{\prime\nu}$};
\node(B2) at (4.5,0) {$Y_\ell^{\prime\nu}$};
\node(C2) at (3,1.75) {$Y^{\prime\nu}_j$};
\node(D2) at (3,-1.75) {$Y^{\prime\nu}_i$};
\end{scope}

\draw[->, dashed, thick] (0.5,0) -- (3,0)
node[midway, above, yshift=1mm] {$\phi^{\nu}$};

\end{tikzpicture}
\caption{A general type II flop of combinatorial K3 surfaces. If $i\neq j$ (and hence $k=\ell$), we are in the situation of Construction~\ref{construction type ii flop pt} and from left to right we go from a surface of class $P$ to one of class $T$. If $i=j$ (and hence $k\neq \ell$), we are in the situation of Construction~\ref{construction type ii flop pt} and go from class $T$ to class $P$.}
\label{figure type ii flop}
\end{figure}
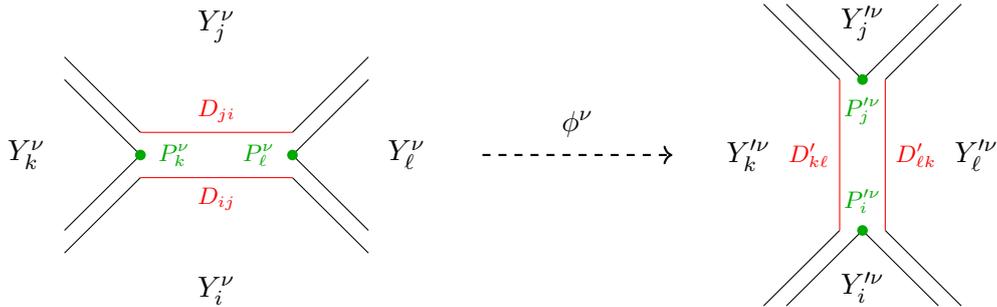

\begin{remark}\label{remark carlson map}
One can generalize these definitions to arbitrary degree, but this requires a further analysis of maximal gluings.
Abstractly, this is a consequence of the proof of \cite[Proposition~2.23]{HL22}. There it is shown that maximality is preserved under type I and II flops as the Carlson map pulls back to the Carlson map; see the arXiv preprint \cite[Proposition~5.18]{AE21}.
\end{remark}

\subsection{Type II flops on curve structures}\label{section type ii flops on curve structures}

The effect of a type II flop on the curve structure is a bit more sophisticated than for type I flops as the following example of such a flop of the surface $Y_P$ shows.

\begin{proposition}\label{proposition type ii flop yp}
A type II flop of $Y_P$ produces the surface $Y'\in \CK2$ obtained from $Y_T$ by successively flopping two curves from each nonspecial component into the special component. Note that there is a unique way of doing this.
\end{proposition}
\begin{proof}
Let us start with some notation. Let $\phi:Y_P \ratl Y'$ denote a type II flop as in the statement. We write $Y_P=Y_1\cup Y_2\cup Y_3$ and $Y'=Y_1' \cup Y_2' \cup Y_3'$ for the irreducible components where the numbering is chosen such that $\phi$ restricts to a birational map $Y_i \ratl Y_i'$. With the notation from Notation~\ref{notation double curves}, we assume that the type II flop is performed in the curve pair $D_{23}$, $D_{32}$. That is, these curves are blown down on $Y_2$ and $Y_3$, and then two exterior (i.e. non-interior) special points are blown up on $Y_1$ to give anticanonical curves $D'_1$, $D'_2$ on $Y_1'$. 

By \cite{HL22}\footnote{This is not explicit in op. cit., but follows easily from the description of elementary modifications in \cite[Section 1.1]{HL22}. This was noted in the proof of \cite[Corollary~6.8]{HL22}.}, the surface $Y'$ is of class $T$. Let us first compute the squares of the anticanonical vertices on $Y'$. On the nonspecial component $Y'_2$ of $Y'$, we denote by $C'$ the anticanonical nodal rational curve. Let $\pi: Y_2\to Y'_2$ be the blow-up given by the restriction of $\phi$. This is the blow-up in the node of the nodal curve $C'$ so that $\pi^*C' = C + 2E$ where $C=D_{21}$ denotes the strict transform of $C'$ and $E=D_{23}$ denotes the exceptional curve. We compute
\[
{C'}^2 = (\pi^*{C'})^2 = (C + 2E)^2=C^2 + 4 C.E + 4E^2 = -1 + 8 - 4=3.
\]
By symmetry, the same holds for the anticanonical curve on $Y'_3$. The squares of the curves in the anticanonical cycle on the special component are given as follows. The two blown up curves correspond to $(-1)$-curves and the other two anticanonical vertices are $(-3)$-curves by the fact that the squares of the glued curves have to add up to zero as soon as one of them is nodal. This last claim is obvious for $Y_T$ and is preserved under type I flops, so it holds for our surface $Y'$. For \emph{simple} normal crossing surfaces, this would also follow from the triple point formula, see \cite[p.~2]{FS86}.

From analyzing the curve structure on $\gothY_1$, see Figure~\ref{figure gothy1 curve structure}, we see that there is a unique way\footnote{This is the case $(n_1,n_2)=(-3,-3)$ in Section 7.2 of \cite{HL22}.} of producing a model with these anti-canonical squares. Since a curve structure uniquely determines the combinatorial K3 surface by Theorem~\ref{theorem curve structures determine surface}, we conclude the proof.
\end{proof}

\begin{figure}[!ht]%
\begin{tikzpicture}[scale=0.65, >=stealth]

  \foreach \i in {0,...,8} {
    \ifnum\i=0
      \node[special] (a\i) at (\i,0) {};
    \else
      \ifnum\i=8
        \node[special] (a\i) at (\i,0) {};
      \else
        \node[vertex] (a\i) at (\i,0) {};
      \fi
    \fi
  }
  \node[vertex] (b) at (4,1) {};
  \draw (a0) -- (a1) -- (a2) -- (a3) -- (a4) -- (a5) -- (a6) -- (a7) -- (a8);
  \draw (a4) -- (b);
  \draw[red, thick] (6.5,-0.5) rectangle (8.5,0.5);

  \draw[->, thick] (9,0) -- (12,0) node[midway, above] {type II flop};

  \foreach \i in {0,...,6} {
    \ifnum\i=0
      \node[special] (c\i) at (13+\i,0) {};
    \else
      \node[vertex] (c\i) at (13+\i,0) {};
    \fi
  }
  \node[vertex] (d) at (17,1) {};
  \draw (c0) -- (c1) -- (c2) -- (c3) -- (c4) -- (c5) -- (c6);
  \draw (c4) -- (d);

  \foreach \i in {0,...,8} {
    \ifnum\i=0
      \node[special] (A\i) at (\i,-3) {};
    \else
      \ifnum\i=8
        \node[special] (A\i) at (\i,-3) {};
      \else
        \node[vertex] (A\i) at (\i,-3) {};
      \fi
    \fi
  }
  \node[vertex] (B) at (4,-2) {};
  \draw (A0) -- (A1) -- (A2) -- (A3) -- (A4) -- (A5) -- (A6) -- (A7) -- (A8);
  \draw (A4) -- (B);

  \draw[->, thick] (9,-3) -- (12,-3) node[midway, above] {type II flop};

<\foreach \i in {0,...,8} {
  \ifnum\i=0
    \node[special] (C\i) at (13+\i,-3) {};
  \else
    \ifnum\i=8
      \node[special] (C\i) at (13+\i,-3) {};
    \else
      \node[vertex] (C\i) at (13+\i,-3) {};
    \fi
  \fi
}
  \node[vertex] (D) at (17,-2) {};
  \draw (C0) -- (C1) -- (C2) -- (C3) -- (C4) -- (C5) -- (C6) -- (C7) -- (C8);
  \draw (C4) -- (D);
  \node[bluewhite] (E) at ($(D)+(-1,0)$) {};
  \node[bluefill] (F) at ($(D)+(-2,0)$) {};
  \node[bluewhite] (G) at ($(D)+(1,0)$) {};
  \node[bluefill] (H) at ($(D)+(2,0)$) {};
  \draw (D) -- (E) -- (F);
  \draw (D) -- (G) -- (H);
\end{tikzpicture}

\caption{The behavior of the curve structures under a type II flop. }%
\label{figure type ii flop curve structure}%
\end{figure}
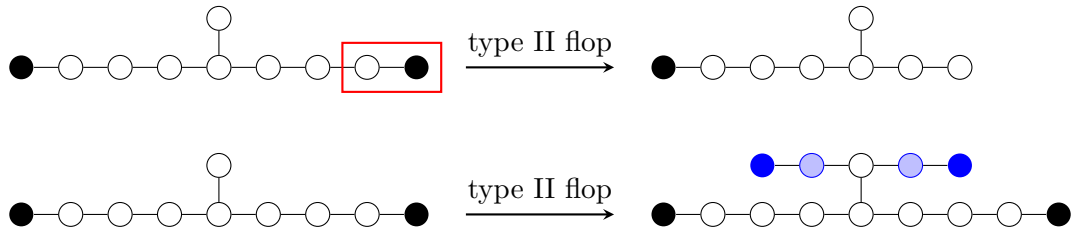

Figure~\ref{figure type ii flop curve structure} illustrates the type II flop $Y_P\ratl Y'$. On each of two components of $Y_P$, a double curve is blown down. These components become nonspecial components in the resulting class $T$ surface $Y'$. The curves adjacent to the exceptional curves in the curve structure become square zero curves on the corresponding components of $Y^{\prime\nu}$ and are no longer part of the curve structure. The third component of $Y_P$, call it $Y_1$,  becomes the special component $Y_1'$ of $Y'$. Here, the two singular points of the anticanonical cycle on $\gothY_2$ are blown up. In the curve structure of $Y'$, they become the blue anticanonical vertices in the bottom right graph of Figure~\ref{figure type ii flop curve structure}. Lemma~\ref{lemma zero curve on gothy2} describes a fibration for which there are uniquely determined smooth rational curves of square zero passing through each blow-up point. They become $(-1)$-curves under the blow-up and become elements in the curve structure of the special component of $Y'$. These are the blue shaded curves in the figure.

\section{Bases and cones}\label{section bases and cones}

The purpose of this section is to provide a description of the cone of curves (and hence also the nef cone) of a combinatorial K3 surface explicit enough to be computable in practice. We will show in Proposition~\ref{proposition cone generators} that the cone of curves of such a surface $Y$ is generated by the curves in the augmented curve structure. The nef cone is obtained as the dual of this cone in $\Pic(Y^\nu)$ intersected with $\Pic(Y)$. 

\subsection{Type I flops on curve structures}\label{section type i flops on curve structures}

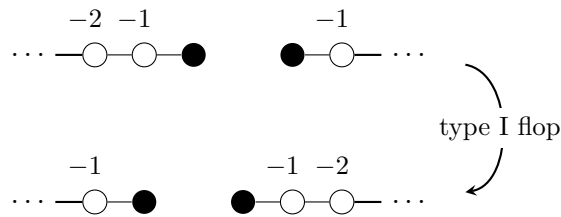
\begin{figure}[!ht]%
\begin{tikzpicture}[scale=0.65, >=stealth]

	
  \foreach \i in {0,...,2} {
    \ifnum\i=2
      \node[special] (a\i) at (\i,0) {};
    \else
      \node[vertex] (a\i) at (\i,0) {};
    \fi
  }
  \draw (a0) -- (a1) -- (a2);
  \draw[thick] ($(0,0)+(-0.8,0)$) -- (a0);
  \node[left] at ($(0,0)+(-0.8,0)$) {$\ldots$};

  \foreach \i in {0,1} {
    \ifnum\i=0
      \node[special] (b\i) at (4+\i,0) {};
    \else
      \node[vertex] (b\i) at (4+\i,0) {};
    \fi
  }
  \draw (b0) -- (b1);
  \draw[thick] (b1) -- ($(b1)+(0.8,0)$);
  \node[right] at ($(b1)+(0.8,0)$) {$\ldots$};

  \node[above, yshift=2mm, xshift=-1.3mm, font=\small] at (a0) {$-2$};
  \node[above, yshift=2mm, xshift=-1.3mm, font=\small] at (a1) {$-1$};
  \node[above, yshift=2mm, xshift=-1.3mm, font=\small] at (b1) {$-1$};

  \begin{scope}[yshift=-3cm]
  \foreach \i in {0,...,1} {
    \ifnum\i=1
      \node[special] (c\i) at (\i,0) {};
    \else
      \node[vertex] (c\i) at (\i,0) {};
    \fi
  }
  \draw (c0) -- (c1);
  \draw[thick] ($(0,0)+(-0.8,0)$) -- (c0);
  \node[left] at ($(0,0)+(-0.8,0)$) {$\ldots$};

  \foreach \i in {0,...,2} {
    \ifnum\i=0
      \node[special] (d\i) at (3+\i,0) {};
    \else
      \node[vertex] (d\i) at (3+\i,0) {};
    \fi
  }
  \draw (d0) -- (d1) -- (d2);
  \draw[thick] (d2) -- ($(d2)+(0.8,0)$);
  \node[right] at ($(d2)+(0.8,0)$) {$\ldots$};

  \node[above, yshift=2mm, xshift=-1.3mm, font=\small] at (c0) {$-1$};
  \node[above, yshift=2mm, xshift=-1.3mm, font=\small] at (d1) {$-1$};
  \node[above, yshift=2mm, xshift=-1.3mm, font=\small] at (d2) {$-2$};
  \end{scope}

\draw[->, thick, bend left=80] ($(b1)+(2.5,-0.2)$) to ($(d2)+(2.5,0.2)$) 
node[font=\small, fill=white, inner sep=3pt] at ($($(b1)+(3.2,0)$)!0.5!($(d2)+(3.2,0)$)$) {type I flop};
\end{tikzpicture}

\caption{The behavior of the curve structures under a type I flop.}%
\label{figure type i flop curve structure}%
\end{figure}

The effect of a type I flop on the curve structure is easy to describe. We flop in a $(-1)$-curve $b$ meeting an anticanonical vertex $d$ (necessarily with multiplicity $1$). This anticanonical vertex is glued to an anticanonical vertex $d'$, possibly on a different component.
During the flop, the vertex $b$ disappears from the diagram and reappears as a vertex $b'$ of square $(-1)$ next to the anticanonical vertex $d'$, see Figure~\ref{figure type i flop curve structure}. For all vertices meeting $b$ before the flop, both their self-intersection and their intersection multiplicity with $d$ increase by one. Conversely, all vertices meeting $d'$ before the flop will meet $b'$ after the flop, their self-intersection and their intersection multiplicity with $d'$ will decrease by one. 

Note that while in Figure~\ref{figure type i flop curve structure} only one white vertex meets $b$, $d$, and $d'$ respectively, the description we just gave is valid in full generality. This is a direct consequence of the definition of a type I flop in Definition~\ref{def:type II flops final} and of a curve structure in Definition~\ref{definition curve structure}.

\subsection{Curves on combinatorial K3 surfaces}

Recall that the vertices in the curve structure are generators of the rational Picard lattice by \cite[Proposition~4.17]{HL22}. One can do a little better in degree $2$.

\begin{lemma}\label{lemma curve structure is z basis}
Let $Y\in \CK2$. Then the curves in the curve structure form a $\Z$-basis of $\Pic(Y^\nu)$.
\end{lemma}
For the proof and for the remainder of this section, recall the notion of an $n$-blow-up from Definition~\ref{def n blow up}.

\begin{proof}
First, we will show that the claim is stable under type I flops. Let $\phi:Y' \ratl Y$ be a type I flop. On the irreducible components, this is either a blow-up, a blow-down, or an isomorphism in which case there is nothing to show. We will show that the claim for $Y$ follows componentwise from the claim for $Y'$. 
Along the blown-down component, say $\phi:Y_1' \to Y_1$, this is obvious since the image of a $\Z$-basis for $\Pic(Y'_1)$ will constitute a $\Z$-basis for $\Pic(Y_1)$. On the blown up component, say $\phi^{-1}:Y_2 \to Y_2'$, by the description of Section~\ref{section type i flops on curve structures} the exceptional curve is added to the curve structure $\Gamma_{Y_2}$ and the other vertices of $\Gamma_{Y_2}$ correspond to the strict transforms of the vertices in $\Gamma_{Y'_2}$. Thus, together these vertices still form a $\Z$-basis for $\Pic(Y_2)$.

Thanks to Theorem~\ref{theorem type i flops to minus one model}, we are left to prove the claim for $Y_P$ and $Y_T$ respectively their irreducible components. We thus have to show for each $i =1,2,4$ that 
\[
\Pic(\gothY_i) = \Z\left\langle C_v \mid v \in \Gamma_{\gothY_i}\right\rangle.
\]
Let us start with $\gothY_1$. By Example~\ref{example geometric description gothy1}, it is the $8$-blow-up of $\P^2$ in an inflection point of a nodal cubic. Thus, the Picard group of $\gothY_1$ is generated by the exceptional curves and the strict transform of a line  in $\P^2$. These are precisely the curves in the curve structure.

Recall from Example~\ref{example geometric description gothy2} that $\gothY_2$ is obtained as a $(3,3)$-blow-up of the Hirzebruch surface $\F_2$. Thus, the Picard group of $\gothY_2$ is generated by the exceptional curves and the strict transform of a fiber and the $(-2)$-curve on $\F_2$. Again, these are precisely the curves in the curve structure.

The surface $\gothY_4$ is a $(1,1,1,1)$-blow-up of $\P^1\times \P^1$ with its toric boundary, see Figure~\ref{figure gothy4 curve structure}. The Picard group is generated by the four exceptional curves and the strict transforms of the fibers of the two projections on $\P^1\times\P^1$. The latter are the inner white vertices of the curve structure, so that indeed $\Pic(\gothY_4)$ is generated over $\Z$ by the vertices in the curve structure.
\end{proof}

\begin{remark}
Note that already by \cite[Proposition~4.17]{HL22}, a combinatorial K3 surface $Y\in \CK{2d}$ of degree $2d$ has Picard rank $\rho(Y)=20+d$. Further, $\Pic(Y)$ embeds into $\Pic(Y^\nu)$ by pullback and the image can be described by the sequence \eqref{eq maximality sequence}.
\end{remark}

\subsection{The cone of curves}\label{section cone of curves}

In order to compute the nef cone of a surface $Y \in \CK{2d}$, it is crucial to compute the cone of curves of $Y^\nu$. The latter is the direct sum of the cone of curves of its components. 

\begin{proposition}\label{proposition cone generators}
Let $Y\in\CK{}$ and let $Y^\nu$ be its normalization. Then the cone $\NEbar(Y^\nu)$ is generated by all curves in the augmented curve structure, i.e. the curve structure plus the anticanonical vertices. In particular, $\NEbar(Y^\nu)$ is rational polyhedral. 
\end{proposition}
\begin{proof}
The statement is verified component by component, so we denote $Y_1 \subset Y^\nu$ a connected component. Let us first assume that $Y_1$ has Picard rank $\rho(Y_1)\geq 3$. In this case, by \cite[Theorem~2.1]{AL11}, we know that the cone $\NEbar(Y_1)$ is generated by the closure of the classes of curves of negative square. Note that the assumption $\kappa(-K_{Y_1})\geq 0$ is satisfied, as $Y_1$ comes with an anticanonical cycle. It thus suffices to show that all curves of negative square show up in the augmented curve structure. This is true by Corollary~\ref{corollary negative curves in curve structure} below.

We turn to the case where $Y_1$ has Picard rank at most $2$. Such a surface is obtained as a birational contraction of a surface of the same class having Picard rank at least~$3$. Let $\pi:Y_1' \to Y_1$ be such a contraction. Hence, the claim holds for $Y_1'$ by what was proven so far. By definition of the augmented curve structure (see Definition~\ref{definition curve structure}), we obtain the one of $Y_1$ from the one of $Y_1'$ by pushforward along $\pi$. Hence, given a curve $C \subset Y_1$, we can write the class of $\pi^{-1}(C) \subset Y_1'$ as a linear combination with nonnegative coefficients of the curves in the curve structure $\Gamma_{Y'_1}^a$. Then the class of $C$ is equal to the pushforward of this linear combination, which, by what we have just said and the definition of a curve structure, is a linear combination 
with nonnegative coefficients of the curves in the curve structure~$\Gamma_{Y_1}^a$. The claim follows.
\end{proof}

\begin{lemma}\label{lemma hulek liese 4.9}
Let $Y_1$ be a component of the normalization of a surface in $\CK{}$ and denote by $D_1, \ldots, D_m$ the components of its anticanonical cycle. Put \mbox{$n_i:=\max(0,D_i^2+1)$} and $n:=(n_1,\ldots,n_m)$.  Further, put $\gothn_i:=\max(0,-1-D_i^2)$ and $\gothn:=(\gothn_1,\ldots,\gothn_m)$. Then there are morphisms
\begin{equation}\label{eq n-blowup}
\xymatrix{
&\wt Y_1 \ar[dl]_p\ar[dr]^q&\\
Y_1 \ar@{-->}[rr]&& \gothY_m\\
}
\end{equation}
where $p$ is the $n$-blow-up of $Y_1$ and $q$ is the $\gothn$-blow-up of $\gothY_i$.
\end{lemma}
\begin{proof}
This is contained in the proof of \cite[Lemma~4.9]{HL22}.
\end{proof}

For later use, we isolate the following statement from the proof of Proposition~\ref{proposition cone generators}.

\begin{corollary}\label{corollary negative curves in curve structure}
Let $Y_1$ be a component of a surface in $\CK{}$. Then all square-negative curves on $Y_1^\nu$ appear in the curve structure. 
\end{corollary}
\begin{proof}
Using Lemma~\ref{lemma hulek liese 4.9}, we reduce to the case where $Y_1$ is the $n$-blow-up of some~$\gothY_i$. Indeed, if the statement holds on a blow-up, it holds on the surface itself by pushforward. Let $C\subset Y_1$ be a curve of negative square and let us denote by $D$ the anticanonical cycle. Then either $C \subset D$ or, by our reduction step and \cite[Proposition~4.5]{HL22}, it is in the curve structure.
\end{proof}

\begin{remark}
Proposition~\ref{proposition cone generators} looks harmless, but it is actually quite surprising. Recall that the cone of curves of a rational surface need not be rational polyhedral. In fact, already for the blow-up of $\P^2$ in $\geq 9$ general points it is not since there are infinitely many $(-1)$-curves by a result of Nagata \cite[Theorem~4a]{Nag60}; see also \cite[Theorem~2.7]{CLO16} and references therein. The above proposition on the other hand applies to arbitrarily high Picard rank surfaces. The reason is that we are not blowing up in general points but instead in rather special points, the interior special points.
\end{remark}

\begin{remark}\label{remark computing nef cones}
Let $Y\in\CK{}$. By \eqref{eq maximality sequence}, we know how to describe the Picard group $\Pic(Y)$ inside $\Pic(Y^\nu)$. Proposition~\ref{proposition cone generators} allows us to compute the nef cone $\Nef(Y^\nu)$ as the dual of the cone of curves. 
One can algorithmically compute the ray generators of a rational polyhedral cone, given generators for its dual cone. This is well known and implemented, for example, in SageMath \cite{sagemath}. Finally, this allows to compute $\Nef(Y)$ as
\[
\Nef(Y) = \Nef(Y^\nu)\cap \Pic(Y) \subset \Pic(Y^\nu).
\]
\end{remark}

\begin{remark}
In order to compute the Mori fan in degree $2d$, one needs to compute nef cones of models $\sY\to S$ of the DNV family in degree $2d$. If $Y$ is the central fiber of $\sY\to S$, we explained in Remark~\ref{remark computing nef cones} how to compute its nef cone. Moreover, by the maximality property the restriction $\Pic(\sY)\to \Pic(Y)$ is an isomorphism, see \cite[Definitions 2.15 and~2.20]{HL22}. It is not difficult to see that this isomorphism preserves nefness so that this gives a way of computing the Mori fan explicitly.\end{remark}

\section{Flops on Picard}\label{section flops on picard}

This section is the core of the paper. The maps described here are completely explicit and thus allow one to compute the Mori fan on a computer. We know by Lemma~\ref{lemma curve structure is z basis} that the vertices in the curve structure provide a basis of the Picard group of the normalization of a combinatorial K3 surface. For each flop $\phi:Y \ratl Y'$ of type I or II, we define a map $\phi_c:\Pic(Y^\nu)\to \Pic(Y^{\prime\nu})$ between the corresponding Picard groups of the normalizations, see Definitions~\ref{definition type i flop on picard of normalizations} and~\ref{definition type ii flop on picard of normalizations}. It is referred to as the \emph{combinatorial pushforward}. The main results of this section are Theorems~\ref{theorem type i flop on picard} and~\ref{theorem type ii flop on picard}, which link the combinatorial pushforward to the pushforward along a flop of the corresponding models of the DNV family, see also Section~\ref{section flops motivation}. We prove the main theorem separately for type~I flops and type II flops. 

\subsection{Motivation}\label{section flops motivation}

Let us give some motivation for what comes. Here, we assume familiarity with models of the DNV family in the sense of \cite[Definition~2.22]{HL22}. 

Let $\Phi:\sY \ratl \sY'$ be a (type I or II) flop between models of the DNV family with central fiber $\phi:Y \ratl Y'$. Since the restrictions $\Pic(\sY) \to \Pic(Y)$, $\Pic(\sY') \to \Pic(Y')$ are isomorphisms by \cite[Definitions~2.15 and~2.20]{HL22}, the question arises how to describe $\Phi_*:\Pic(\sY) \to \Pic(\sY')$ under these restriction isomorphisms, i.e. in terms of the combinatorial K3 surfaces $Y$ and $Y'$. It clearly does \emph{not} coincide with $\phi_*$ since the latter is not even an isomorphism.

Consider the diagram
\begin{equation}\label{eq lift to pic normalization}
\xymatrix{
\Pic(\sY) \ar[d]_{\Phi_*} \ar[r]^{\iota^*} & \Pic(Y)\ar@{-->}[d]^{\textcolor{red}{\textbf{?}}} \ar@{^(->}[r]^{\nu^*}& \Pic(Y^\nu)\ar[d]^{\phi_c}\\
\Pic(\sY') \ar[r]^{\iota^{\prime*}} & \Pic(Y') \ar@{^(->}[r]^{\nu^*} & \Pic(Y^{\prime\nu}) \\
}
\end{equation}
where $\iota$, $\iota'$ denote the inclusions of the central fibers. 
The statement of Theorem~\ref{theorem type i flop on picard} resp. Theorem~\ref{theorem type ii flop on picard} is that the outer diagram is commutative in the degree $2d=2$ case. In particular, the combinatorial pushforward $\phi_c$ restricts to $\iota^{\prime*} \circ \Phi_* \circ \left(\iota^*\right)^{-1}$ on $\Pic(Y)$. This is remarkable for at least two reasons. 
\begin{enumerate}
	\item Firstly, this gives a fairly neat description of the geometric operation $\iota^{\prime*} \circ \Phi_* \circ \left(\iota^*\right)^{-1}$ by means of combinatorial data.
	\item Secondly, it is surprising that $\Phi_*$ has an extension to the Picard groups of the normalizations. There seems to be no apparent geometric reason for this (we emphasize that $\phi_c\neq \phi_*$).
\end{enumerate}

\subsection{Type I Flops on Picard}\label{section flops on picard i}

Recall from Lemma~\ref{lemma curve structure is z basis} that for $Y\in\CK{}$, the Picard group $\Pic(Y^\nu)$ has a canonical $\Z$-basis given by the vertices of the associated curve structure. We refer to this basis as the \emph{standard basis} of $Y$.

\begin{setup}\label{setup type i flop}
Let $\phi:Y \ratl Y'$ be a type I flop between combinatorial K3 surfaces. We denote by $\sB_\std, \sB_\std'$ the standard bases of $Y, Y'$. Then there are vertices $v_\flop \in \sB_\std$, $v_\flop'\in \sB_\std'$ such that (the curve corresponding to) $v_\flop$ is contracted under $\phi$ and (the curve corresponding to) $v_\flop'$ is contracted under $\phi^{-1}$. We will denote $(\cdot)':\sB_\std \to \sB_\std'$ be the bijection sending $v_\flop$ to $v_\flop'$ and all other vertices to their strict transforms. We denote the intersection pairings on $Y, Y'$ by $(\cdot,\cdot), (\cdot,\cdot)'$. 
\end{setup}

\begin{definition}
In the Setup~\ref{setup type i flop}, we call $v \in \sB_\std$ an \emph{old neighbor} if $v\neq v_\flop$ and $(v,v_\flop) \neq 0$. Similarly, we call $v' \in \sB_\std'$ a  \emph{new neighbor} if $v'\neq v_\flop'$ and $(v',v'_\flop)' \neq 0$. By abuse of terminology, we will call a vertex $v'\in\sB_\std'$ an old neighbor respectively $v\in\sB_\std$ a new neighbor if they correspond to old respectively new neighbors under $\phi$. 
\end{definition}

\begin{example}\label{example neighbors}
Consider the type I flop from Figure~\ref{figure type i flop curve structure}. Call the white vertices from left to right $a$, $b$, $c$ in the top row and $a'$, $b'$, $c'$ in the bottom row. Then $a$ and $a'$ are ``old neighbors'' and $c$ and $c'$ are ``new neighbors''. In terms of the curves they represent, the 'neighbors' (old or new) correspond exactly to those curves that meet the curve corresponding to $v_\flop$ on the surface $Y$. 
\end{example}

Note that we have \emph{disjoint} unions
\begin{equation}\label{eq decomposition standard basis type i flop}
\sB_\std = \sB_c \cup \sB_\old \cup \sB_\new \cup \{v_\flop\} \quad \textrm{ and } \quad \sB_\std' = \sB_c' \cup \sB'_\old \cup \sB'_\new \cup \{v'_\flop\} 
\end{equation}
which are preserved by $\phi$. Here, the subscripts $\old$ and $\new$ denote old respectively new neighbors and $\sB_c, \sB_c'$ are defined to be the remaining vertices (i.e. neither neighbors nor the exceptional curve). Here, the subscript \emph{c} stands for \emph{complement}.

\begin{definition}\label{definition type i flop on picard of normalizations}
Let $\phi:Y \ratl Y'$ be a type I flop of combinatorial K3 surfaces. We define the \emph{combinatorial pushforward} morphism $\phi_c:\Pic(Y^\nu)\to \Pic(Y^{\prime\nu}) $ as follows:
\begin{equation}\label{eq definition isometry on pic type i}
\phi_c(v):= \begin{cases} -v_\flop' & \textrm{ if } v=v_\flop\\[2pt]
v'+v'_\flop & \textrm{ if } v\in \sB_\new \cup \sB_\old\\[2pt]
 v' & \textrm{ otherwise.}\end{cases}
\end{equation}
\end{definition}

\begin{proposition}
In the Setup~\ref{setup type i flop}, the combinatorial pushforward $\phi_c$ is an isometry.
\end{proposition}
\begin{proof}
It suffices to check that the pairings between vectors of the standard bases are preserved. This is obvious if one of the vectors is in $\sB_c$. Clearly, also the square of $v_\flop$ is preserved. To handle the other cases, we will assume that $v,w \in \sB_\std$ are neighbors and $v',w'\in \sB_\std'$ are their strict transforms. Then
\[
\begin{aligned}
\left(\phi_c(v_\flop),\phi_c(v)\right)' &= (-v_\flop',v'+v_\flop')' = 1 - (v_\flop',v')' = \begin{cases} 0 & \textrm{ if } v'\in \sB'_\new \\1 & \textrm{ if } v'\in \sB'_\old  \end{cases} \\
& = (v_\flop,v).
\end{aligned}
\]
One computes similarly that $(\phi_c(v),\phi_c(w))'=(v,w)$.
\end{proof}

For the comparison of the geometric pushforward $\Phi_*$ and the combinatorial pushforward $\phi_c$, we give a description of the threefold flop. We refer to \cite[p.~13]{FM83} for the proof resp. construction.

\begin{lemma}\label{lemma threefold flop}
Let $\Phi:\sY \ratl \sY'$ be a type I or type II flop that flops a curve $B \subset Y$ to a curve $B'\subset Y'$. Then there is a commutative diagram as follows.
\[
\xymatrix{
&\Bl_B\sY \ar[r]^\isom \ar[dl]_p&\Bl_{B'}\sY'\ar[dr]^q&\\
\sY \ar@{-->}[rrr]^\Phi&&& \sY'
}
\]
We will thus identify $\wt\sY:=\Bl_B\sY = \Bl_{B'}\sY'$. Then the exceptional loci of $p$ and $q$ coincide and are given by an irreducible divisor $E\isom \P^1\times \P^1$. Under this identification, $p$ and $q$ are identified with the projections to the factors. Moreover, $B'$ is the image under $q$ of a fiber of $p\vert_E$. Then the linear map $\Phi_*$ is given by $q_*p^*$.\qed
\end{lemma}

Before we come to the main result of this section, we will treat a special case, which is very instructive.

\begin{example}\label{example geometric description pushforward}
Let us consider a type I flop $\Phi:\sY_P \ratl \sY'$ where $\sY_P$ is the model of the DNV family having $Y_P$ as a central fiber. We denote $Y'$ the central fiber of $\sY'$. Then locally around the flopped curve, the behavior of the flop on the central fibers is described by Figure~\ref{figure type i flop curve structure}. If $a,b,c$ denote the white vertices in the top curve structure (corresponding to curves $A,B,C$ on $Y_P$) and $a',b',c'$ denote the white vertices in the bottom curve structure (corresponding to curves $A',B',C'$ on $Y'$), then $A$ and $B+C$ define classes in $\Pic(Y_P) \isom \Pic(\sY_P)$ and the combinatorial pushforward takes the following values:
\[
\phi_c(a) = a'+b' \quad \textrm{ and } \quad \phi_c(b+c)=c'.
\]
We thus have to find an effective divisor $D \subset \sY_P$ whose restriction $D_0$ to the central fiber is $A$ resp. $B+C$ and show that in this case $D':=\Phi_*D$ restricts to $A'+B'$ resp. $C'$ on the central fiber of $\sY'$. In both cases, $D$ is unique: the classes in $\Pic(\sY_P)$ extending $a$ resp. $b+c$ both have square $(-2)$ and thus are either effective or anti-effective. The latter alternative is excluded, since an effective class must specialize to an effective class. Let us denote by $D$ an effective divisor on $\sY_P$ representing these classes, one sees that it must be irreducible. Indeed, if there were several irreducible components, all would restrict to $(-2)$ curves on the K3 surface $(\sY_P)_\eta$ so they cannot restrict to $A$ or $B+C$ on the central fiber. 

We will now describe $D'$ using Lemma~\ref{lemma threefold flop} and the notation introduced there. For $D_0=A$, the universal property of the blow-up tells us that $p^*D = \wt D$ is just the strict transform and $\wt D \isom \Bl_{A \cap B} D$. Notice that $A \cap B$ is a reduced point since $D.B = A.B =1$ where the first intersection is taken in $\sY_P$ and the second one in the component of $Y_P$ containing $A$ and $B$. We see that $D'=\Phi_*D=q_*\wt D$ restricts to $A'+B'$ on the central fiber, as claimed.

For $D_0=B+C$, one argues similarly and finds that $D'$ restricts to $C'$. Note that in this case, $p^*D = \wt D + E$, but the exceptional divisor $E$ is contracted under $q_*$. 
\end{example}

\begin{example}\label{example intersection description pushforward}
While the geometric description of $\Phi_*$ is not too difficult in the special case of Example~\ref{example geometric description pushforward}, it can get rather involved in general, mainly due to the various possibilities for the curve structure in a neighborhood of the flopped curve. Instead, we will undertake a different approach based on intersection products and the projection formula. For this, let $\wt K \subset \wt\sY:=\Bl_B\sY_P$ be an integral curve and denote $K:=p_*\wt K$ and $K':=q_*\wt K$. By the projection formula, for every divisor $D$ on $\sY_P$ we have  
\[
D.K = p^*D.\wt K = \wt D.\wt K + aE.\wt K
\]
where $a$ is the multiplicity of $B$ in $D_0$. The idea now is that $\phi_c(D_0)$ and $D'_0$ coincide if and only if they define the same linear form on the space of $1$-cycles. For this, we choose $\wt K$ such that $K'$ runs through the curves in the curve structure. For all curves $K'$ different from $B'$ in the curve structure of $\sY'_0$, there is a unique choice for $\wt K$ and $K$. To obtain $K'=B'$, we choose $\wt K=:\wt B$ to be the diagonal in the exceptional divisor $E=\P^1\times \P^1$. Then $K=B$. 

Let us revisit the result from Example~\ref{example geometric description pushforward} in this setup. For this purpose, let $D$ be the unique prime divisor such that $D_0=A$. Then $p^*D=\wt D$, but $q^*D' = \wt D + E$. For all curves $K$ in the curve structure different from $A,B,C$, we have $\wt K.E=0$. Recall from \cite[p.~13]{FM83} that $E\vert_E = \sO_{\P^1\times\P^1}(-1,-1)$. We compute
\[
\begin{aligned}
D'.A' &= p^*D.\wt A + E.\wt A = D.A + 1 = -1 = \phi_c(a).a',\\
D'.B' &= p^*D.\wt B + E.\wt B = D.B - 2 = -1 = \phi_c(a).b',\\
D'.C' &= p^*D.\wt C + E.\wt C = D.C +1  = 1 = \phi_c(a).c',\\
\end{aligned}
\]
so that indeed $\phi_c$ and $\Phi_*$ coincide on $D$. Notice that by the projection formula, the intersection number of $D$ with a curve can always be computed on the normalization of a component of $Y$ to which the curve lifts. A similar computation settles the case $D_0=B+C$. 
\end{example}

Clearly, Examples~\ref{example geometric description pushforward} and~\ref{example intersection description pushforward} do not only prove the claim for a type I flop of $\sY_P$, but in fact for every type I flop from a model whose curve structure near to the flopped curve looks like the upper one in Figure~\ref{figure type i flop curve structure}. 
Now we turn to the general comparison between the combinatorial and the geometric pushforward as in \eqref{eq lift to pic normalization} for type I flops.

\begin{theorem}\label{theorem type i flop on picard}
Let $\Phi:\sY \ratl \sY'$ be a type I flop between models of the DNV family and let $\phi:Y \ratl Y'$ denote the induced type I flop between the central fibers. Then the diagram 
\[
\xymatrix{
\Pic(\sY) \ar@{^(->}[r]\ar[d]_{\Phi_*} & \Pic(Y^\nu) \ar[d]^{\phi_c}\\
\Pic(\sY') \ar@{^(->}[r] & \Pic(Y^{\prime\nu}) \\
}
\]
is commutative where the horizontal isomorphisms are given by restriction to the central fiber and pullback to the normalization and $\phi_c$ is the combinatorial pushforward.
\end{theorem}
\begin{proof}
We use Lemma~\ref{lemma threefold flop} together with the notation introduced there. 
Since both maps $\Phi_*$ and $\phi_c$ are linear, it suffices to verify the claim on ample divisors since those generate the Picard group. Also, one can always replace such a divisor by a multiple. Let $D\subset \sY$ be a sufficiently ample divisor. In this case, $h^0(\sY_s,\sO_\sY(D)_s) = \chi(\sO_\sY(D)_s)$ by Serre vanishing and this is independent of the point $s\in S$. Thus, we can choose $D$ general in its linear system so that we may assume it does not contain the exceptional curve but only meets it transversally in a finite number of points. We denote $D':=\Phi_*D$ and write the restrictions $D_0$ and $D_0'$ to the central fibers $\sY_0=Y$ and $\sY'_0=Y'$ with respect to the standard bases as
\begin{equation}\label{eq pushforward in standard basis}
\begin{aligned}
D_0&\sim\sum_{v\in \sB_\old} a_v A_v + \sum_{v\in \sB_\new} b_v B_v + \sum_{v\in \sB_c} c_v C_v + \lambda B \quad \textrm{ resp. } \\
D'_0&\sim\sum_{v\in \sB'_\old} a_v A'_v + \sum_{v\in \sB'_\new} b_v B'_v + \sum_{v\in \sB'_c} c_v C'_v + \lambda' B',\\
\end{aligned}
\end{equation}
up to linear equivalence. Here, the decomposition is according to \eqref{eq decomposition standard basis type i flop}. Note that since $\Phi$ is an isomorphism at the generic point of the $A_v, B_v,C_v$, the coefficients at those curves do not change from $D_0$ to $D_0'$ (where we identified the bases $\sB_\std$, $\sB'_\std$ as in Setup~\ref{setup type i flop}). From \eqref{eq pushforward in standard basis}, we read off the combinatorial pushforward of $D$ as
\begin{equation}\label{eq combinatorial pushforward in standard basis}
\phi_c(D_0)= \sum_{v\in \sB'_\old} a_v A'_v + \sum_{v\in \sB'_\new} b_v B'_v + \sum_{v\in \sB'_c} c_v C'_v + \left(\sum_v a_v + \sum_v b_v - \lambda \right) \cdot B'.
\end{equation}

Comparing \eqref{eq pushforward in standard basis} and \eqref{eq combinatorial pushforward in standard basis}, we conclude the proof by Lemma~\ref{lemma coefficients exceptional divisor}.
\end{proof}

\begin{lemma}\label{lemma coefficients exceptional divisor}
In the notation of \eqref{eq pushforward in standard basis}, we have 
\[
\lambda + \lambda' = \sum_{v\in \sB_\old} a_v + \sum_{v\in \sB_\new} b_v.
\]
\end{lemma}
\begin{proof}
Using \eqref{eq pushforward in standard basis}, we write the pullbacks of $D,D'$ to $\wt\sY$ as
\begin{equation}\label{eq pullbacks}
p^*D = \wt D  \qquad \textrm{ and }\qquad q^*D' = \wt D + \mu' E,
\end{equation}
where $\wt D$ is the strict transform and $E \subset \wt\sY$ is the exceptional divisor (for both $p$ and~$q$). Note that since $D$ is a general ample divisor, it does not contain $B$, so the pullback coincides with the strict transform. Let us consider a fiber $F'$ of $p\vert_E: E \to B$. Then $p_*F'=0$ and $q_*F'=B'$. We use the projection formula to compute
\[
\begin{aligned}
0 &= D.p_*F' = p^*D.F' = \wt D.F'\\
&= q^*D'.F' -\mu' E.F' \\
&= D'.B' + \mu' \\
&= -\lambda' + \sum_v b_v + \mu',\\
\end{aligned}
\]
where in the last step we used \eqref{eq pushforward in standard basis}. Let us denote by $F$ now a fiber of $q\vert_E:E \to B'$. Then we find analogously
\[
\begin{aligned}
0 &= D'.q_*F = q^*D'.F = \wt D.F + \mu' E.F\\
&= p^*D.F + \mu' E.F' \\
&= D.B - \mu' \\
&= -\lambda + \sum_v a_v - \mu',\\
\end{aligned}
\]
Putting together the two equations, we obtain
\[
\lambda' - \sum_v b_v = \mu' = -\lambda + \sum_v a_v
\] 
\end{proof}

\begin{corollary}
In the Setup~\ref{setup type i flop}, the map $\phi_c$ sends $\Pic(Y)$ into $\Pic(Y')$.\qed
\end{corollary}

\subsection{Type II Flops on Picard}\label{section flops on picard ii}

We will now define the combinatorial pushforward for a type II flop. We could define the sought-for morphism $\phi_c$ from \eqref{eq lift to pic normalization} on the standard basis $\sB_\std$ from Setup~\ref{setup type i flop} consisting of all curves in the curve structure, but it turns out that the description in another basis is easier. We will add the anticanonical vertices in exchange for the vertices adjacent to them. More precisely:

\begin{setup}\label{setup type ii flop}
Let $\phi:Y \ratl Y'$ be a type II flop between combinatorial K3 surfaces. We denote by $\sB_\std, \sB_\std'$ the standard bases of $Y, Y'$. Let $d_i$, $d_i'$ for $i=1,2$ be the anticanonical vertices on $Y$, $Y'$ which are blown down under $\phi$ resp. $\phi^{-1}$. Let $v_i$, $v_i'$ for $i=1,2$ be the unique vertices of the curve structure having nonzero intersection with $d_i$, $d_i'$. We define the bases
\[
\sB_{II}:= \left(\sB_\std \ohne \{v_1, v_2\}\right) \cup \{d_1,d_2\}
\]
of $\Pic(Y^\nu)$ respectively 
\[
\sB_{II}':= \left(\sB_\std' \ohne \{v_1', v_2'\}\right) \cup \{d_1',d_2'\}
\]
of $\Pic(Y^{\prime\nu})$. Taking the strict transform gives a canonical bijection 
\[
\sB_\std \ohne \{v_1, v_2\} \to \sB_\std \ohne \{v_1', v_2'\}.
\]
We extend this to a bijection $(\cdot)': \sB_{II} \to \sB_{II}'$ by sending $d_i \mapsto d_i'$.
\end{setup}

\begin{remark}
Note that there is a small ambiguity here: there is no canonical correspondence between the $d_i$ and the $d_i'$. Contrary to what the notation suggests, they are not canonically ordered, so that once we have chosen an ordering (which we will tacitly do), sending $d_1 \mapsto d_2'$ and $d_2\mapsto d_1'$ would be an equally reasonable choice. We do not pay attention to this, since this issue only plays a role on the normalization and we are ultimately interested in the glued surfaces. Suppose for the sake of the argument that $Y'$ is of class $T$ (and thus $Y$ is of class $P$). Then the curves $d_i'$ for $i=1,2$ are glued to one another under the normalization map. While the combinatorial pushforward map on $\Pic(Y^{\prime\nu})$ to be defined below will depend on the numbering of the $d_i'$, its restriction to $\Pic(Y')$ will not. This is reflected graphically in the fact that the curves $d_i'$ are exchanged under the graph automorphism of the augmented curve structure that fixes the upper row and the white $(-2)$-vertex in the lower row of Figure~\ref{figure gothy4 curve structure} and swaps the two leftmost vertices with the two rightmost vertices in the lower row.
\end{remark}

\begin{definition}\label{definition type ii flop on picard of normalizations}
Let $\phi:Y \ratl Y'$ be a type II flop of combinatorial K3 surfaces. We define the \emph{combinatorial pushforward} morphism $\phi_c:\Pic(Y^\nu)\to \Pic(Y^{\prime\nu}) $ as follows:
\begin{equation}\label{eq definition isometry on pic type ii}
\phi_c(v):= \begin{cases} -v' & \textrm{ if } v\in \{d_1,d_2\}\\ v' & \textrm{ otherwise.} \\\end{cases}
\end{equation}
\end{definition}

\begin{theorem}\label{theorem type ii flop on picard}
Let $\Phi:\sY \ratl \sY'$ be a type II flop between models of the DNV family and let $\phi:Y \ratl Y'$ denote the induced type II flop between the central fibers. Then the diagram 
\[
\xymatrix{
\Pic(\sY) \ar@{^(->}[r]\ar[d]_{\Phi_*} & \Pic(Y^\nu) \ar[d]^{\phi_c}\\
\Pic(\sY') \ar@{^(->}[r] & \Pic(Y^{\prime\nu}) \\
}
\]
is commutative where the horizontal isomorphisms are given by restriction to the central fiber and pullback to the normalization and $\phi_c$ is the combinatorial pushforward.
\end{theorem}
\begin{proof}
The argument is similar to the case of a type I flop, however, the curve structure is a bit simpler. This manifests itself in our choice of basis, cf. Setup~\ref{setup type ii flop}. Let $B\subset \sY$ be the curve to be flopped and denote $B_1, B_2$ the curves in the normalization $Y^\nu$ of $Y=\sY_0$ that map to $B$. Then $B_1, B_2$ are basis vectors in $\sB_{II}$ and all other basis vectors do not intersect $B$. On those curves, there is nothing to prove. 

Observe that $B_1+B_2 \in \Pic(Y)$ and this is a square $(-2)$ class. As in Example~\ref{example geometric description pushforward}, there is a unique divisor $D$ on $\sY$ that restricts to $D_0=B_1+B_2$ on the central fiber. We only have to show that $\phi_c(D_0)=(\Phi_*D)_0$. We use Lemma~\ref{lemma threefold flop} and the notation therein and write $D':=\Phi_*D$ and
\[
p^*D = \wt D + \lambda E \quad \textrm{ and } \quad q^*D' = \wt D + \lambda' E
\]
where $\wt D$ is the strict transform. As before, we consider a fiber $F'$ of $p\vert_E: E \to B$. Then the projection formula gives us
\[
\begin{aligned}
0 &= D.p_*F' = \wt D.F' + \lambda E.F'\\
&= q^*D'.F' + (\lambda- \lambda') E.F' \\
&= D'.B' + (\lambda'-\lambda)\\
\end{aligned}
\]
Similarly, if $F$ denotes a fiber of $q\vert_E:E \to B'$, we find 
\[
\begin{aligned}
0 &= D'.q_*F = \wt D.F + \lambda' E.F\\
&= p^*D.F + (\lambda'-\lambda) E.F \\
&= D.B - (\lambda'-\lambda).\\
\end{aligned}
\]
Putting together both equations, we obtain
\[
D.B = \lambda'-\lambda = - D'.B'.
\]
Again, since all other curves in $\sB_{II}$ intersect $D$ respectively $D'$ trivially, this shows that $D'_0=\phi_c(D_0)$ and we conclude the proof.
\end{proof}




\end{document}